\documentclass[11pt,reqno,a4paper]{amsart}

\usepackage{amsmath,amsfonts,amsthm,mathrsfs,amssymb,cite}
\usepackage[usenames]{color}
\usepackage{graphicx}
\usepackage{latexsym} 
\usepackage{enumerate}
\usepackage{hyperref}
\usepackage{xcolor}
\hypersetup{
  colorlinks,
  linkcolor={blue!80!black},
  urlcolor={blue!80!black},
  citecolor={blue!80!black}
}

\theoremstyle{plain}
\newtheorem{theorem}{Theorem}[section]
\newtheorem*{theorem*}{Theorem}
\newtheorem{lemma}[theorem]{Lemma}
\newtheorem{corollary}[theorem]{Corollary}
\newtheorem{proposition}[theorem]{Proposition}
\newtheorem*{proposition*}{Proposition}

\newtheorem*{conjecture*}{Conjecture}

\theoremstyle{definition}
\newtheorem{definition}[theorem]{Definition}

\theoremstyle{remark}

\newtheorem*{remark*}{Remark}

\DeclareMathOperator{\supp}{supp}

\newcommand{\abs}[1]{\left\lvert #1 \right\rvert}
\newcommand{\norm}[1]{\left\lVert #1 \right\rVert}

\newcommand{\R}{\mathbb{R}}
\newcommand{\C}{\mathbb{C}}

\newcommand{\N}{\mathbb{N}}

\renewcommand{\eqref}[1]{\textnormal{(\ref{#1})}}

\numberwithin{equation}{section}

\title[Stability in corner scattering and inverse acoustic
  scattering]{On corners scattering stably and stable shape
  determination by a single far-field pattern}

\author{Emilia~L.K. Bl{\aa}sten} \address{Department of Mathematics
  and Systems Analysis, Aalto University, FI-00076 Aalto, Finland AND
  Division of Mathematics, Tallinn University of Technology,
  Department of Cybernetics, 19086 Tallinn, Estonia}
\email{emilia.blasten@iki.fi}

\author{Hongyu Liu} \address{Department of Mathematics, City
  University of Hong Kong, 83 Tat Chee Ave, Kowloon, Hong Kong SAR}
\email{hongyuliu@cityu.edu.hk}

\begin{document}

\begin{abstract}

In this paper, we establish two sharp quantitative results for the
direct and inverse time-harmonic acoustic wave scattering. The first
one is concerned with the recovery of the support of an inhomogeneous
medium, independent of its contents, by a single far-field
measurement. For this challenging inverse scattering problem, we
establish a sharp stability estimate of logarithmic type when the
medium support is a polyhedral domain in $\mathbb{R}^n$, $n=2,3$. The
second one is concerned with the stability for corner scattering. More
precisely if an inhomogeneous scatterer, whose support has a corner,
is probed by an incident plane-wave, we show that the energy of the
scattered far-field possesses a positive lower bound depending only on
the geometry of the corner and bounds on the refractive index of the
medium there. This implies the impossibility of approximate
invisibility cloaking by a device containing a corner and made of
isotropic material. Our results sharply quantify the qualitative corner
scattering results in the literature, and the corresponding proofs involve much
more subtle analysis and technical arguments. As a significant
byproduct of this study, we establish a quantitative Rellich's theorem
that continues smallness of the wave field from the far-field up to
the interior of the inhomogeneity. The result is of significant
mathematical interest for its own sake and is surprisingly not yet
known in the literature.


\end{abstract}




\maketitle

\section{Introduction}\label{sec:Intro}

  In this paper, we are concerned with the direct and inverse problems
  associated with time-harmonic acoustic scattering described by the
  Helmholtz system as follows. Let $k\in\mathbb{R}_+$ be a wavenumber
  of the acoustic wave, signifying the frequency of the wave
  propagation. Let $V\in L^\infty(\mathbb{R}^n)$, $n=2, 3$, be a
  potential function. $V(x)$ signifies the material parameter of the
  medium at the point $x$ and it is related to the refractive index in
  our setting. We assume that $\mathrm{supp}(V)\subset B_R$, where
  $B_R$ is a central ball of radius $R\in\mathbb{R}_+$ in
  $\mathbb{R}^n$. That is, the inhomogeneity of the medium is
  supported inside a given bounded domain of interest. The
  inhomogeneous medium is often referred to as a {\it scatterer}.

  \subsection*{Wave model}
  A common model in probing with waves is to send an incident wave
  field to interrogate the medium $V$. The latter perturbs the former
  to create a total wave field. We let $u^i$ and $u$, respectively,
  denote the incident and total wave fields. The former is an entire
  solution to the Helmholtz equation $(\Delta+k^2)u^i=0$ and $u$
  satisfies
  \begin{equation}\label{eq:Helm1}
    \big(\Delta + k^2(1+V)\big) u = 0,
  \end{equation}
  in $\R^n$. Moreover, the scattered wave $u^s=u-u^i$ satisfies the
  Sommerfeld radiation condition
  \begin{equation}\label{eq:Helm2}
    \abs{x}^{\frac{n-1}{2}} \big(\partial_r - ik\big) u^s \to 0,
  \end{equation}
  uniformly with respect to the angular variable $\theta:=x/|x|$ as
  $r:=\abs{x}\to\infty$. Here, $\partial_r$ is the derivative along
  the radial direction from the origin. The radiation condition
  implies the existence of a far-field pattern. More precisely there
  is a real-analytic function on the unit-sphere at infinity $A_{u^i}:
  \mathbb{S}^{n-1}\to\C$ such that
  \begin{equation}\label{eq:Helm3}
    u(r\theta) = u^i(r\theta) + \frac{e^{ikr}}{r^{(n-1)/2}}
    A_{u^i}(\theta) + \mathcal{O} \Big( \frac{1}{r^{n/2}} \Big)
  \end{equation}
  uniformly along the angular variable $\theta$. This function is
  called the \emph{far-field pattern} or \emph{scattering amplitude}
  of $u$.

  \subsection*{Problem statements}
  The inverse scattering problem that we are concerned with is to
  recover $V$ or its shape, namely the support, from the knowledge of
  $A_{u^i}(\theta)$. A related direct scattering problem of practical
  importance is to investigate under what circumstance one would have
  $A_{u^i}(\theta)\equiv 0$. The former serves as a prototype model to
  many inverse problems arising from scientific and technological
  applications \cite{CK,Isa,Uhl}. The direct scattering problem is
  related to a significant engineering application, \emph{invisibility
    cloaking} (cf. \cite{GKLU4,GKLU5,Uhl2}). We next briefly discuss
  some related progress and open questions in the literature on both
  of these two topics.

  \subsection*{Shape determination}

 Concerning the inverse scattering problem des\-cri\-bed above, we are
 mainly interested in recovering the shape of the inhomogeneous
 scatterer, namely its support. Furthermore, we consider the recovery
 in the formally-determined case with a single far-field measurement,
 that is, the scattering amplitude produced from a single wave
 incidence. The shape determination by minimal or optimal measurement
 data remains a longstanding open problem in inverse scattering theory
 \cite{CK,Isa}. It has been conjectured that one can uniquely
 determine the shape of an impenetrable scatterer by a single
 far-field measurement. Significant progress has been achieved in
 recent years in uniquely recovering impenetrable polyhedral
 scatterers by minimal numbers of far-field measurements; see
 \cite{AR,CY,LZ,J18} for related unique recovery results, and
 \cite{LPRX,Rondi08} for optimal stability estimates. However, very
 little is known in the literature concerning the shape determination
 of a penetrable medium scatterer, independent of its content, by a
 single far-field measurement. Recently, based on the qualitative
 corner scattering result by one of the authors of the current article
 \cite{BPS}, it is show in \cite{HSV} that if two penetrable
 scatterers $V$ and $V'$ produce the same scattering amplitude for any
 single incident wave, namely $A_{u^i}=A'_{u^i}$ then the difference
 of the supports of $V$ and $V'$, namely $\mathrm{supp}(V) \,
 \triangle \, \mathrm{supp}(V') := \big( \mathrm{supp}(V) \setminus
 \mathrm{supp}(V') \big) \cup \big( \mathrm{supp}(V') \setminus
 \mathrm{supp}(V) \big)$, cannot have a corner of the type that
 appeared in the papers on corner scattering that shall be discussed
 in what follows. This means, in particular, that in the set of convex
 polygonal or cuboidal penetrable scatterers the far-field pattern
 produced by sending any single incident wave uniquely determines the
 shape and location of the scatterer.

  In this article, we sharply quantify the aforementioned uniqueness
  result on the shape determination by a single far-field
  pattern. More precisely, we establish logarithmic estimates in
  determining the shape of a medium scatterer supported in a 2D
  polygonal or 3D cuboidal domain. In essence given two such
  penetrable mediums $V$ and $V'$ and a common incident wave $u^i$, if
  the far-field patterns of the scattered waves $u-u^i$ and $u'-u^i$
  are $\varepsilon$-close to one another then the supporting polytopes
  of $V$ and $V'$ are $\varphi(\varepsilon)$-close in the sense of
  Hausdorff distance. Here $\varphi$ is of double-logarithmic
  type. For precise statements see Section~\ref{sect:results}.

  \subsection*{Far-field lower bound and relation to invisibility}
  Concerning the direct scattering problem described earlier, it is
  proved in \cite{BPS} that if $A_{u^i} \equiv 0$ for a single
  incident wave $u^i$ then the support of $V$ cannot have a $90^\circ$
  corner in $\R^n$. In \cite{PSV}, it is further shown that under
  similar conditions, the support of $V$ cannot have a conical
  corner\footnote{With the exception of a discrete set of opening
    angles in 3D under which nothing is known so far.} in $\R^2$ or
  $\R^3$.
  
  The above qualitative results indicate that a penetrable corner
  scatters every incident wave non-trivially. This has significant
  implications for invisibility cloaking, which is a moniker for
  technologies that cause an object, such as a spaceship or an
  individual, to be partially or wholly invisible with respect to
  light or other wave detection. Blueprints for achieving invisibility
  with respect to electromagnetic waves via the use of the
  artificially engineered \emph{metamaterials} were recently proposed
  in \cite{GLU2,Leo,PenSchSmi}. These materials are anisotropic and
  singular. The same idea has also been developed for acoustic waves
  using acoustic metamaterials; see \cite{ctchan} and the references
  cited therein. Due to its practical importance, the mathematical
  study on invisibility cloaking has received significant attentions
  in the last decade; see \cite{GKLU4,GKLU5,J53,J23,J33,Uhl2} and the references
  therein.

  The singularity of the metamaterials for perfect cloaking poses
  sever difficulties to practical realisation. In order to avoid the
  singular structures, various regularised approximate cloaking
  schemes have been proposed. They make use of non-singular
  metamaterials and we refer to the survey paper \cite{LU} and the
  references cited therein. However, these regularised metamaterials
  are still nearly singular in the sense that they depend on an
  asymptotic regularisation parameter and as the regularisation
  parameter tends to zero, the material become singular. It is of
  scientific interest and practical importance to know whether one can
  achieve invisibility by completely regular materials.

  Our results imply not only that cloaking by regular materials is
  impossible, but also so is approximate cloaking, if there is a
  corner on the cloaking device. Indeed, in
  Theorem~\ref{cornerScatStab} we quantify the corner scattering
  results in \cite{BPS,PSV} by showing that for an inhomogeneous
  medium scatterer supported on a polygon/polyhedron, the energy of
  the scattering amplitude possesses a positive lower bound. We prove
  this for regular isotropic acoustic mediums, and similar results are
  in progress for regular anisotropic acoustic mediums as well as
  electromagnetic mediums. We refer to these results as the stability
  issue of corner scattering. Our study indicates that corners not
  only scatter non-trivially but also in a stable way.

  \subsection*{On a significant byproduct}
  The basis of our proofs is on quantifying the estimates and
  coefficients arising in the proofs of \cite{BPS}. However, as can be expected, it involves much more
  subtle analysis and technical arguments due to the delicate
  analytical and geometrical situation. We postpone the discussion of
  our mathematical arguments to Section~\ref{sect:ideas}. In what
  follows, we would like to comment on a significant by product of the
  current study. In order to establish the sharp stability estimates
  mentioned earlier, we need a quantitative version of the unique
  continuation and Rellich's theorem which is surprisingly not yet
  known in the literature. Our context requires that scattered waves
  be small partly inside the penetrable scatterer. A result proving
  this starting from a small far-field pattern has been overlooked in
  the literature. This problem turns out to be highly non-trivial and
  technical and we believe that this result would find important
  application in other challenging scattering problems. In the sequel,
  we briefly discuss the difficulties of the result achieved.
  
  In scattering theory a vanishing far-field pattern implies that the
  scattered wave is zero outside the scattering object \cite{CK}. This
  follows by unique continuation and Rellich's theorem. Instead, we
  require that a small scattering amplitude means a small scattered
  wave, all the way up to the boundary of the support of the
  scatterer. Despite the innocent look of this sentence there is a lot
  of work to do. The impenetrable case is known in the literature
  \cite{Isakov92, Isakov93, Rondi08, RondiSini, LPRX}. Not so for
  penetrable scatterers. There might be two reasons for this lack of
  results: a) waves behave the same outside a penetrable or
  impenetrable scatterer, and b) typically in showing stability in
  inverse medium scattering, the far-field data are reduced to the
  Dirichlet-Neumann map as in \cite{Sylvester--Uhlmann, Nachman88}.
  We cannot use either conditions.

  Orthogonality relations in corner scattering require an estimate for
  the scattered wave that is valid at the boundary of the
  scatterer. Boundary estimates are completely ignored for
  impenetrable obstacles because boundary conditions are imposed
  \emph{a-priori} there. Secondly, the Dirichlet-Neumann map is badly
  suited for our case since we are interested in \emph{a single}
  incident wave and the associated far-field pattern of the scattered
  wave. Restricting to a single incident wave is also the reason why
  inverse backscattering is still unsolved for general potentials (see
  e.g. \cite{RU1,RU2}). One cannot construct special solutions for
  probing the problem in the single wave incidence case.

  We prove a quantitative unique continuation and Rellich's theorem
  for penetrable scatterers in Section~\ref{sect:ff2scat}. There is a
  major issue compared to the impenetrable case: we do not have a
  boundary condition for the total wave at the boundary of the
  scatterer. We cannot use quantitative unique continuation to
  propagate smallness all the way into the boundary of the convex
  hull, as the associated function stops being real-analytic
  there. Dealing with this issue is the source of the two logarithms
  in our stability estimates.

  \subsection*{Layout}
  The structure of the paper is as follows. We define notation in the
  next section, which helps with stating the main theorems in
  Section~\ref{sect:results}. The proof idea is described in
  Section~\ref{sect:ideas}. The quantitative Rellich's theorem and
  propagation of smallness are proven in
  Section~\ref{sect:ff2scat}. The fundamental integral identity, along
  with estimates for its various terms is shown in
  Section~\ref{sect:orthogonality}. The following one,
  Section~\ref{sect:cgo}, has the precise estimates for the complex
  geometrical optics solutions. Finally after all the ingredients have
  been prepared, the main theorems are proven in
  Section~\ref{sect:proofs}. The appendix contains proofs of technical
  geometrical lemmas.

\section{Notation} \label{sect:notation}

\begin{enumerate}
  \item We use italic letters $P, Q, \ldots$ to denote polytopes,
    fraktura symbols $\mathfrak P, \mathfrak Q, \ldots$ for polyhedral
    cones, and calligraphic symbols $\mathcal P, \mathcal Q, \ldots$
    for spherical cones. This is purely a stylistic choice: all
    symbols are defined in their context,
  \item $B_R = B(0,R)$, $0<R<\infty$: a-priori domain of interest,
    where the scatterers are located in,
  \item $P, P' \subset B_R$: the shape of the penetrable scatterers,
    which are open polytopes,
  \item $d_H(P,P')$: the Hausdorff distance between the sets $P$ and
    $P'$, defined by
    \[
    d_H(P,P') = \max\big( \sup_{x\in P} d(x,P'), \, \sup_{x'\in P'}
    d(x',P') \big),
    \]
  \item $\norm{P}_{T(s,r)}$: a type of norm for the characteristic
    function $\chi_P$. If it is finite, the latter is a multiplier in
    the Sobolev space $H^s_r(\R^n)$. See Definition~\ref{triangDef},
  \item $u^i$: incident wave,
  \item $u$, $u'$: corresponding total waves.
\end{enumerate}

\begin{definition}[Well-posed scattering]
  A potential $V \in L^\infty(\R^n)$ is said to give a
  \emph{well-posed scattering problem} if there is a finite $\mathcal
  S$ such that given any incident plane-wave $u^i(x) =
  \exp(ik\omega\cdot x)$ there is a unique $u\in H^2_{loc}$ such that
  \[
  (\Delta + k^2(1+V))u = 0
  \]
  and the scattered wave $u^s = u - u^i$ satisfies the Sommerfeld
  radiation condition. Moreover it has to have the norm bound
  $\norm{u^s}_{H^2(B_{2R})} \leq \mathcal S$.
\end{definition}

\begin{definition}[Admissible shape] \label{admissibleShape}
  A polytope $P \subset B_R$ is \emph{admissible} if
  \begin{enumerate}
  \item in 2D, it is a bounded open convex polygon, and
  \item in 3D, it is a cuboid, i.e. there is a rigid motion taking $P$
    to ${]{0,a}[} \times {]{0,b}[} \times {]{0,c}[}$ for some
    $a,b,c>0$.
  \end{enumerate}
\end{definition}

\begin{definition}[Admissible contrast] \label{admissibleContrast}
  Given an admissible shape $P\subset B_R$, a function
  $\varphi:\R^n\to\C$ is \emph{admissible} if
  \begin{enumerate}
  \item $\varphi\in C^\alpha$ for some $\alpha>0$ in 2D, and
    $\alpha>1/4$ in 3D,
  \item $\varphi\neq0$ at the vertices of $P$.
  \end{enumerate}
\end{definition}

If the wave-number or the potential is small, $k^2 \norm{V}_\infty <
C_0$, then the Neumann series construction of the total wave shows
directly that there is well-posed scattering. Unique continuation and
Fredholm theory generalises this. For details see Section 8.4 in
\cite{CK}. An alternative approach is by \cite{Hormander}, see for
example the introduction in \cite{HSV}. Note that if $P$ and $\varphi$
are admissible, then $V = \chi_P \varphi$ has well-posed scattering at
any positive frequency $k>0$.

\begin{definition}[Non-vanishing total wave] \label{nonVanishing}
  We say that a potential $V \in L^\infty(B_R)$ produces a
  \emph{non-vanishing total wave} if given any incident plane-wave
  $u^i$ the total wave $u$ vanishes nowhere in $B_R \setminus \supp
  V$.
\end{definition}

We again emphasise that this condition is satisfied for $k$ or
$\norm{V}_\infty$ small enough, but more general situations exist. It
is well-known that the vanishing set (nodal set) of the total field
cannot be too large, however how it relates to a particular potential
is an open problem.

\section{Statement of the stability results} \label{sect:results}

We assume the following a-priori bounds on the potentials. Given any
admissible shape $P$ and function $\varphi$ it is possible to choose
these parameters such that $V = \chi_P \varphi$ satisfies these
bounds.
\begin{definition}[A-priori bounds] \label{aprioriBounds}
  The following two theorems have dimension $n\in\{2,3\}$, wavenumber
  $k>0$ and radius of the domain of interest $R>1$ fixed as a-priori
  parameters. In addition
  \begin{enumerate}
  \item the minimal distance from any vertex of $P$ to a non-adjacent
    edge is at least $\ell$ which we assume at most $1$ for technical
    reasons,
  \item in 2D, $P$ has angles at least $2\alpha_m>0$ and at most
    $2\alpha_M<\pi$,
  \item $\norm{P}_{T(s,r)} \leq \mathcal D$, see
    Definition~\ref{triangDef},
  \item $\norm{\varphi}_{C^\alpha} \leq \mathcal M$,
  \item $\abs{\varphi(x_c)} \geq \mu$ for any vertex $x_c$ of $P$,
  \item if $V$ is required to produce non-vanishing total waves, then
    assume that the infimum of the waves' absolute value in $B_R$ is
    at least $c>0$.
  \end{enumerate}
\end{definition}

\begin{theorem} \label{potSupportStab}
  Let $V, V' \in L^\infty(B_R)$ be potentials of the form $V = \chi_P
  \varphi$, $V' = \chi_{P'} \varphi'$ with $P,P'$ and
  $\varphi,\varphi'$ admissible by Definition~\ref{admissibleShape}
  and Definition~\ref{admissibleContrast}. Moreover assume that $V$
  and $V'$ produce non-vanishing total waves as in
  Definition~\ref{nonVanishing}.

  \smallskip
  Let $\mathfrak h = d_H(P,P')$ be the Hausdorff distance of $P$ and
  $P'$.  Let $u^i(x) = \exp(ik\omega\cdot x)$ be any plane-wave and
  $u^s_\infty, u'^s_\infty$ be the far-field patterns of the scattered
  waves produced by $V$ and $V'$, respectively.

  There are constants $\varepsilon_{min}, C<\infty$ --- which depend
  on the a-priori bounds of Definition~\ref{aprioriBounds} only ---
  and $\gamma = \gamma(\alpha,n,r,s)>0$ such that if
  \[
  \norm{u^s_\infty - u'^s_\infty}_{L^2(\mathbb S^{n-1})} < 
  \varepsilon_{min}
  \]
  then
  \begin{equation} \label{distanceBound}
    \mathfrak h \leq C \left( \ln\ln \frac{\mathcal S}{\norm{
        u^s_\infty - u'^s_\infty }_{L^2(\mathbb S^{n-1})}}
    \right)^{-\gamma}.
  \end{equation}
\end{theorem}

\medskip
We remark that in the following theorem the refractive index function
$\varphi$ is allowed to vanish at the vertices. As long as there is
one corner where it does not vanish, and the scatterer can fit inside
the convex cone generated by that corner, then we can show a lower
bound for the scattering amplitude. {\color{black} We would also like to point out that in Theorem~\ref{potSupportStab}, the scattering potential can actually be required to be H\"older-continuous 
only in an open neighbourhood of its corner, and be $L^\infty$ elsewhere in its support. This can be seen from the corresponding proof of Theorem~\ref{potSupportStab} in what follows. In fact, in the corresponding arguments, the H\"older continuity is only used in a neighbourhood of the corner point. However, in
order to ease the exposition and discussion, we present our study that $\varphi$ is H\"older-continuous in $P$ (resp. $\varphi'$ is H\"older-continuous in $P'$). 

}

\begin{theorem} \label{cornerScatStab}
  Let $V\in L^\infty(B_R)$ be a potential of the form $V = \chi_P
  \varphi$ with $P$ and $\varphi$ admissible by
  Definition~\ref{admissibleShape} and
  Definition~\ref{admissibleContrast}.

  Recall that $\ell$ is a lower bound for the minimal vertex to
  non-adjacent edge distance of $P$. Let $u^i(x) = \exp(ik\omega\cdot
  x)$ be any plane-wave and $u^s_\infty$ be the far-field pattern of
  the scattered wave produced by $V$.

  Then
  \begin{equation}
    \norm{u^s_\infty}_{L^2(\mathbb S^{n-1})} \geq \min\left(
    \frac{\mathcal S}{\exp \exp (C \ell^{-2/\gamma}
      \abs{\varphi(x_c)}^{-2-2/((n+5)\gamma)})}, \, \varepsilon_{min}
    \right).
  \end{equation}
  where the constants $\varepsilon_{min}, C<\infty$ depend only on the
  a-priori parameters of Definition~\ref{aprioriBounds} except for
  $\ell$ or $\mu$, and $\gamma = \gamma(\alpha,n,r,s)>0$ is as in the
  previous theorem.
\end{theorem}

{\color{black} Similar to our remark earlier, the scattering potential can actually be required to be H\"older-continuous only
in an open neighbourhood of the corner, and be $L^\infty$ elsewhere in its support. This remark has interesting implications for composite materials used for cloaking applications whose
material parameters are usually piecewise constants.

}

\section{Idea of the proofs} \label{sect:ideas}
We start describing the proof of stability for scatterer support
probing. After this it is very convenient to show stability of corner
scattering by having the second scatterer identically zero.
Propagation of smallness is the first step.

Let $w = u-u'$ be the difference of the total (and hence scattered)
waves from two potentials $V=\chi_P\varphi$ and
$V'=\chi_{P'}\varphi'$. Its far-field pattern is the difference of the
far-field patterns of $u$ and $u'$, and hence small when proving
stability. We first propagate that smallness into the near-field by an
Isakov-type estimate. After that we propagate it near the scatterers
by a chain of balls argument and then into the scatterers by a
delicate balancing argument using H\"older continuity.

Local issues are dealt with next. Focus on a vertex $x_c\in\partial P$
which makes $d(x_c,P')$ equal to the Hausdorff distance between $P$
and $P'$. Let $P_h = P \cap B(x_c,h)$ for some $h > 0$ small
enough. We have two representations for the integral
\[
\int_{P_h} V(x) u_0(x) u'(x) dx
\]
where $u_0$ is any (possibly nonphysical) solution to
\begin{equation}\label{nonPhysSol}
  \big(\Delta + k^2(1+V)\big)u_0 = 0
\end{equation}
and $u':\R^n\to\C$ is the total wave satisfying $\big(\Delta +
k^2(1+V')\big)u' = 0$ corresponding to the incident wave $u^i$. Near
$P_h$ it is actually a solution to the constant coefficient equation
\begin{equation}\label{totalWaveAsIncident}
  (\Delta + k^2)u'=0
\end{equation}
because $V'=0$ there.

For the first representation we use \eqref{nonPhysSol} and Green's
formula. The total wave $u$ satisfies
\begin{equation}\label{totalWave}
  \big(\Delta + k^2(1+V)\big)u=0.
\end{equation}
Integration by parts in a truncated cone $Q_h$ slightly larger than
$P_h$ gives
\begin{equation}
  k^2 \int_{P_h} V(x) u_0(x) u'(x) dx = -\int_{\partial Q_h} \big(
  u_0\partial_\nu (u-u') - (u-u') \partial_\nu u_0 \big) d\sigma
\end{equation}
by \eqref{totalWaveAsIncident} and \eqref{totalWave}.

For the second representation the a-priori admissibility assumptions
and the real-analyticity of $u'$ near $P_h$ imply the splittings
\begin{align*}
  V(x) &= \varphi(x_c) + \varphi_\alpha(x), && \abs{\varphi_\alpha(x)}
  \leq \norm{\varphi}_{C^\alpha(C_h)} \abs{x-x_c}^\alpha,\\ u'(x) &=
  u'(x_c) + u'_1(x), && \abs{u'_1(x)} \leq \mathcal R \abs{x-x_c}.
\end{align*}

Lastly, we choose $u_0 : \R^n\to\C$ to be a complex geometrical optics
solution
\[
u_0(x) = e^{\rho\cdot(x-x_c)}\big(1+\psi(x)\big)
\]
with $\rho\in\C^n$ such that $\exp(\rho\cdot x)$ decays exponentially
in $P_h$ as $\abs{\rho}\to\infty$. We show that there are $p\geq1$ and
$\beta>0$ such that
\[
\norm{\psi}_{L^p(\R^n)} \leq C \abs{\Im\rho}^{-n/p-\beta} \norm{V}
\]
where $C$ doesn't depend on $\rho$ or $V$ as long as $\abs{\Im\rho}$
is large enough. However here the norm $\norm{V}$ is of new type and
contains information about the geometry of the polytope $P$ and
a-priori parameters related to $\varphi$.

Plug the above function splittings into $\int V u_0 u' dx$ and then
estimate all of these integrals in terms of the norms of $u-u'$,
$\varphi(x_c)$, $\abs{\Re\rho}$ and $h$. After that a choice of
$\abs{\Re \rho}$ proves an upper bound for $d_H(P,P')$ based on the
smallness $\varepsilon$ of the far-field pattern of $u-u'$.

\section{From the far-field to the scatterer} \label{sect:ff2scat}

The classical Rellich's theorem (Lemma 2.11 in \cite{CK}) says that if
the far-field pattern of a scattered wave is zero, then the scattered
wave is identically zero on the unbounded and connected component of
space that's unperturbed by a potential or source term. In this
section we study what is the corresponding quantitative result: namely
having a penetrable scatterer and a far-field pattern whose norm is
small but positive. This kind of question has been studied earlier for
the easier case of impenetrable scatterers by Isakov \cite{Isakov92},
\cite{Isakov93}, and more recently by for example Rondi \cite{Rondi08}
and Liu, Petrini, Rondi, Xiao \cite{LPRX}.

Our strategy in this section is as follows. We first generalise a far-field
to near-field estimate in the style of Isakov \cite{Isakov93} and
Rondi, Sini \cite{RondiSini} to the penetrable scatterer case. Then we
use an $L^\infty$ three-spheres inequality to propagate smallness from
the boundary of $B_{2R}$ to almost the support of the scatterer
$V$. To proceed after that use the H\"older continuity of $w =
u-u'$. This allows the propagation to take the final step, crossing
from outside the support of the potentials into the support. Lastly,
we use an elliptic regularity estimate to see that the same operations
can be done for $w=\nabla(u-u')$.

\subsection*{From the far-field to the near-field}

Here we show that if the far-field patterns $A_{u^i}$, $A'_{u^i}$ of
$u$ and $u'$ are close, then $u$ and $u'$ are close in
$B_{2R}\setminus B_R$.

\begin{lemma}\label{logOptimize}
  Let $A, \varepsilon, \mathscr S > 0$. Then there is a function
  $\ell:\R_+\to\R_+$ such that for
  \[
  f(\varepsilon,\ell) = \left(\frac{\ell}{A}\right)^\ell \varepsilon^2
  + \mathscr{S}^2
  \]
  we have
  \[
  f(\varepsilon, \ell(\varepsilon)) \leq 2 \max(\mathscr{S}^2,
  \varepsilon^2).
  \]
  Moreover, when $\varepsilon < \mathscr S$ we may set
  $\ell(\varepsilon) = \sqrt{2A\ln\frac{\mathscr S}{\varepsilon}}$.
\end{lemma}

\begin{proof}
  If $\varepsilon \geq \mathscr S$ choose $\ell(\varepsilon) = A$.
  Otherwise $\ln(\mathscr S/\varepsilon) > 0$ and we may set $\ell$ as
  in the statement, which implies that
  \[
  \frac{\ell}{A} \ln \frac{\ell}{A} \leq \left( \frac{\ell}{A}
  \right)^2 = \frac{2}{A} \ln \frac{\mathscr S}{\varepsilon}
  \]
  i.e. $(\ell/A)^\ell \leq \mathscr{S}^2/\varepsilon^2$ from which the
  claim follows.
\end{proof}

The following proposition generalises Theorem 4.1 from Rondi and Sini
\cite{RondiSini} to the penetrable scatterer case.

\begin{proposition}\label{FF2NF}
  Assume that $w^s \in H^2_{loc}(\R^n)$ satisfies $(\Delta +
  k^2)w^s=0$ in $\R^n \setminus \overline B(0,R)$ and the Sommerfeld
  radiation condition. Let $B_0>1$, $\mathscr S \geq 0$ and assume the
  a-priori bound $\norm{w^s}_{L^2(B_{2R}\setminus B_R)} \leq \mathscr
  S$.

  Let $\varepsilon = \norm{w^s_\infty}_{L^2(\mathbb S^{n-1})}$ where
  $w^s_\infty$ is the far-field pattern of $w^s$. Then there is a
  constant $\mathscr C>0$ depending only on $k,R,B_0$ such that if
  $\varepsilon < \mathscr C^{-1} \mathscr S$ then
  \[
  \norm{w^s}_{L^2(B_{2B_0R}\setminus B_{B_0R})} \leq \mathscr C
  \mathscr S B_0^{-\frac{1}{2} \sqrt{2ekR \ln (\mathscr
      S/\varepsilon)}}.
  \]
  However if not, then $\norm{w^s}_{L^2(B_{2R}\setminus B_R)} \leq
  \mathscr C \varepsilon$.
\end{proposition}

\begin{proof}
  By the assumptions on $w^s$ it is well known that there is a
  sequence $b_j > 0$, $j=0,1,\ldots$ such that its far-field pattern
  $w^s_\infty$ satisfies
  \[
  \norm{w^s_\infty}^2_{L^2(\mathbb S^{n-1})} = \sum_{j=0}^\infty b_j^2
  \]
  and the function itself has
  \[
  \norm{w^s}_{L^2(S(0,r))}^2 = \frac{\pi}{2} \sum_{j=0}^\infty b_j^2 k
  r \abs{ H_{j+(n-2)/2}^{(1)}(kr)}^2
  \]
  for any $r > R$. Here $H^{(1)}_\nu$ is a Hankel function of first
  kind and order $\nu$.

  Let $j_0 \in \{0,1,2,\ldots\}$ and $B_0 > 1$. Then
  \begin{align}
  &\norm{w^s}_{L^2(S(0,r))}^2 = \frac{\pi}{2} \sum_{j=0}^{j_0} b_j^2 k
    r \abs{H^{(1)}_{j+(n-2)/2} (kr)}^2 \notag\\ &\qquad \quad +
    \frac{\pi}{2} \sum_{j=j_0+1}^\infty b_j^2 k r
    \frac{\abs{H^{(1)}_{j+(n-2)/2}(kr)}^2}{\abs{H^{(1)}_{j+(n-2)/2}(kr/B_0)}^2}
    \abs{H^{(1)}_{j+(n-2)/2}(kr/B_0)}^2 \notag\\ &\qquad \leq
    \frac{\pi}{2} k r \max_{0\leq j \leq j_0}
    \abs{H^{(1)}_{j+(n-2)/2}(kr)}^2 \norm{w^s_\infty}_{L^2(\mathbb
      S^{n-1})}^2 \notag\\ &\qquad \quad + B_0 \sup_{j>j_0}
    \frac{\abs{H^{(1)}_{j+(n-2)/2}(kr)}^2}{\abs{H^{(1)}_{j+(n-2)/2}(kr/B_0)}^2}
    \norm{w^s}_{L^2(S(0,r/B_0))}^2 \label{sphericalHarmonics}
  \end{align}
  by the two formulas above. By Corollary 3.8 from Rondi and Sini
  \cite{RondiSini} we see that if $0<z_1 \leq z_2 < \infty$ then there
  is $C = C(z_1,z_2) < \infty$ such that
  \begin{equation}\label{HankelEst1}
    \abs{H^{(1)}_0(z)}^2 \leq C^2 \leq C^2 \frac{4}{\pi e z}
  \end{equation}
  and
  \begin{equation}\label{HankelEst2}
    C^{-2} \frac{4}{\pi e z} \left( \frac{2\nu}{ez} \right)^{2\nu-1}
    \leq \abs{H^{(1)}_\nu(z)}^2 \leq C^2 \frac{4}{\pi e z} \left(
    \frac{2\nu}{ez} \right)^{2\nu-1}
  \end{equation}
  for $z_1 \leq z \leq z_2$ and $\nu \in \{\frac{1}{2}, \frac{2}{2},
  \frac{3}{2}, \ldots\}$. We will integrate the formula above for
  $\norm{w^s}_{L^2(S(0,r))}^2$ along the segment $r\in
       {[{B_0R,2B_0R}]}$, and so the minimal value of $kr/B_0$ will be
       $z_1 := k R > 0$, and the maximal value of the larger $kr$
       shall be $z_2 := 2B_0kR < \infty$.

  Write $\nu_0 = j_0 +(n-2)/2$ and assume that $j_0$ is large enough
  that $\nu_0 \geq ez_2/2 = eB_0kR$ and $\nu_0>1$. These assumptions
  imply that $2\nu_0 \geq ez$ when $z_1\leq z\leq z_2$, and thus also
  \[
  \abs{H_0^{(1)}(z)}^2 \leq \frac{4 C^2}{\pi e z} \leq \frac{4
    C^2}{\pi e z} \left( \frac{2\nu_0}{ez} \right)^{2\nu_0-1}, \quad
  z_1 \leq z \leq z_2.
  \]
  Next, if $1/2 \leq \nu \leq ez/2$ and it is a half-integer, we have
  \[
  \abs{H_\nu^{(1)}(z)}^2 \leq \frac{4 C^2}{\pi e z} \left(
  \frac{2\nu}{ez} \right)^{2\nu-1} \leq \frac{4 C^2}{\pi e z} \leq
  \frac{4 C^2}{\pi e z} \left( \frac{2\nu_0}{e z} \right)^{2\nu_0-1},
  \quad z_1 \leq z \leq z_2.
  \]
  On the other hand if $ez/2 \leq \nu \leq \nu_0$ then
  \[
  \abs{H_\nu^{(1)}(z)}^2 \leq \frac{4 C^2}{\pi e z} \left(
  \frac{2\nu}{ez} \right)^{2\nu-1} \leq \frac{4 C^2}{\pi e z} \left(
  \frac{2\nu_0}{ez} \right)^{2\nu_0-1}, \quad z_1 \leq z \leq z_2
  \]
  because the function $\nu \mapsto (2\nu/(ez))^{2\nu-1}$ defined on
  $\R_+$ is increasing when $\ln 2\nu - \ln z - (2\nu)^{-1} \geq
  0$. This is true when $2\nu \geq e z$ and $\nu \geq 1/2$. In
  conclusion, we can estimate
  \[
  \abs{H_{j+(n-2)/2}^{(1)}(kr)}^2 \leq \frac{4C^2}{\pi ekr} \left(
  \frac{2\nu_0}{ekr} \right)^{2\nu_0-1}, \quad B_0R \leq r \leq 2B_0R
  \]
  in \eqref{sphericalHarmonics} when $0\leq j \leq j_0$. Then, using
  the two Hankel function estimates \eqref{HankelEst1} and
  \eqref{HankelEst2} and recalling that $B_0^{1-2\nu} \leq
  B_0^{1-2\nu_0}$ when $\nu\geq\nu_0$ and $B_0\geq1$, we can continue
  estimating \eqref{sphericalHarmonics} with
  \begin{equation} \label{afterHankel}
    \norm{w^s}_{L^2(S(0,r))}^2 \leq \frac{2 C^2}{e} \left(
    \frac{2\nu_0}{e k r} \right)^{2\nu_0-1}
    \norm{w^s_\infty}_{L^2(\mathbb S^{n-1})}^2 + C^4 B_0^{1-2\nu_0}
    \norm{w^s}_{L^2(S(0,r/B_0))}^2
  \end{equation}
  whenever $B_0R \leq r \leq 2B_0R$.

  \medskip
  Next, we integrate \eqref{afterHankel} by $\int_{B_0R}^{2B_0R}
  \ldots dr$ to get
  \begin{align*}
    &\norm{w^s}_{L^2(B_{2B_0R}\setminus B_{B_0R})}^2 \\ &\qquad \leq
    \frac{C^2 R B_0}{e} \frac{1}{\nu_0-1} \left( 1 -
    \frac{1}{2^{2\nu_0-2}} \right) \left( \frac{2\nu_0}{ekRB_0}
    \right)^{2\nu_0-1} \norm{w^s_\infty}_{L^2(\mathbb S^{n-1})}^2
    \\ &\qquad\quad + C^4 B_0^{2-2\nu_0} \norm{w^s}_{L^2(B_{2R}
      \setminus B_R)}^2
  \end{align*}
  where we have denoted $0$-centred discs of radius $\ell$ by
  $B_\ell$. Use the shorthand $\varepsilon =
  \norm{w^s_\infty}_{L^2(\mathbb S^{n-1})}$ and recall from the
  proposition statement that $\mathscr S \geq \norm{w^s}_{L^2(B_{2R}
    \setminus B_R)}$. Since $\nu_0 \in \frac12\N$ and $\nu_0>1$ we
  have $\lvert(1-e^{2-2\nu_0})/(\nu_0-1)\rvert \leq 2$. Thus
  \[
  \norm{w^s}_{L^2(B_{2B_0R} \setminus B_{B_0R})}^2 \leq \max\left(
  \frac{2C^2R}{e}, C^4 \right) \frac{1}{B_0^{2\nu_0-2}} \left(
  \left(\frac{2\nu_0}{ekR}\right)^{2\nu_0-1} \varepsilon^2 + \mathscr
  S^2\right)
  \]
  when $\nu_0\in\frac12\N$ with $\nu_0>1$ and $\nu_0 \geq eB_0kR$.

  \medskip We are now ready to fix $\nu_0$. Let
  \begin{equation} \label{ellDef}
    \ell = \sqrt{2ekR\ln (\mathscr S/\varepsilon)}, \qquad \nu_0 =
    \lfloor \ell \rfloor/2.
  \end{equation}
  If $\nu_0 < \max(3/2, eB_0kR)$ then
  \[
  \max(3, 2eB_0kR) > 2\nu_0 = \lfloor \ell \rfloor > \ell - 1
  \]
  and so
  \[
  \frac{\mathscr S}{\varepsilon} = \exp\left( \frac{\ell^2}{2ekR}
  \right) < \exp \left( \frac{ (1+\max(3,2eB_0kR))^2 }{2ekR} \right)
  \]
  which implies $\norm{w^s}_{L^2(B_{2R}\setminus B_R)} \leq \mathscr C
  \varepsilon$. On the other hand this would follow even more directly
  if $\mathscr S < \varepsilon$. The other case, namely $\nu_0 \geq
  \max(3/2, eB_0kR)$ and $\mathscr S \geq \varepsilon$, implies in
  particular that
  \[
  \left( \frac{2\nu_0}{ekR} \right)^{2\nu_0-1} \leq \left(
  \frac{\ell}{ekR} \right)^{2\nu_0-1} \leq \left( \frac{\ell}{ekR}
  \right)^\ell
  \]
  because $\ell \geq \lfloor\ell\rfloor = 2\nu_0$ and $2\nu_0 \geq
  2eB_0kR \geq ekR$, as well as
  $2\nu_0-1\geq0$. Lemma~\ref{logOptimize} implies
  \[
  \norm{w^s}_{L^2(B_{2B_0R\setminus B_0R})}^2 \leq \max
  \left(\frac{2C^2R}{e}, C^4\right)
  \frac{2\mathscr{S}^2}{B_0^{2\nu_0-2}} \leq \max
  \left(\frac{2C^2R}{e}, C^4\right) \frac{2\mathscr S^2}{B_0^{\ell-3}}
  \]
  because $2\nu_0 = \lfloor\ell\rfloor \geq \ell-1$ and $\mathscr S
  \geq \varepsilon$ in this final case. The final claim follows from
  the choice of $\ell$ in \eqref{ellDef}.
\end{proof}

\begin{corollary}\label{usedFF2NF}
  Let $w^s \in H^2_{loc}(\R^n)$ satisfy $(\Delta+k^2)w^s=0$ in
  $\R^n\setminus \overline B(0,R)$ and the Sommerfeld radiation
  condition at infinity. Let $w^s_\infty$ be its far-field pattern.
  
  Let $\mathscr S \geq 0$ and assume the a-priori bound
  $\norm{w^s}_{L^2(B_{2R}\setminus B_R)} \leq \mathscr S$. Denote
  $\varepsilon = \norm{w^s_\infty}_{L^2(\mathbb{S}^{n-1})}$. Let $A$
  be a domain such that $\overline A \subset B_{2R} \setminus
  \overline B_R$. Then, for any smoothness index $r\in\N$, there are
  constants $c,C>0$ depending only on $k,r,R,A$ such that
  \[
  \norm{w^s}_{H^r(A)} \leq C \max \left(\varepsilon, \mathscr S e^{-c
    \sqrt{\ln (\mathscr S /\varepsilon)}} \right).
  \]
\end{corollary}

\begin{proof}
  Elliptic interior regularity is the main tool to prove the
  claim. Firstly, if $f \in \mathscr{S}'(\R^n)$ then
  \[
  \norm{f}_{H^{s+2}(\R^n)} = \norm{ (1+k^2) f - (\Delta+k^2)
    f}_{H^s(\R^n)}
  \]
  for any $s\in\R$. Let $B_0 > 1$ be such that $\overline A \subset
  \Omega := B_{2R} \setminus \overline B_{B_0R}$. If $\varphi \in
  C^\infty_0(\Omega)$ we have
  \[
  \norm{(\Delta+k^2)(\varphi w^s)}_{H^s(\R^n)} = \norm{2\nabla\varphi
    \cdot \nabla w^s + w^s \Delta \varphi}_{H^s(\R^n)} \leq C_\varphi
  \norm{w^s}_{H^{s+1}(\Omega)}.
  \]
  Here $w^s$ was extended by zero outside of $\Omega$. Let $\Omega'
  \subset \Omega$ be a subdomain a positive distance from the boundary
  of $\Omega$. Now, if we have $\varphi \equiv 1$ on $\Omega'$, then
  \begin{align*}
    &\norm{w^s}_{H^{s+2}(\Omega')} \leq \norm{\varphi
      w^s}_{H^{s+2}(\R^n)} \leq (1+k^2)C_\varphi
    \norm{w^s}_{H^s(\Omega)} + C_\varphi \norm{w^s}_{H^{s+1}(\Omega)}
    \\ &\qquad\leq C_{k,\varphi} \norm{w^s}_{H^{s+1}(\Omega)}.
  \end{align*}
  by the two equations above.

  Next, the proposition implies 
  \[
  \norm{w^s}_{L^2(\Omega)} \leq \mathscr C \max \left(\varepsilon,
  \mathscr S B_0^{-\frac{1}{2} \sqrt{2ekR \ln (\mathscr
      S/\varepsilon)}} \right)
  \]
  directly. Given $r\in\N$ take a sequence $A=\Omega_r \subset
  \Omega_{r-1} \subset \cdots \subset \Omega_0 = \Omega$ of sets whose
  boundaries are a positive distance apart. Also, take a sequence of
  smooth cutoff functions $\varphi_j \in C^\infty_0(\Omega_j)$ such
  that $\varphi_j \equiv 1$ on $\Omega_{j+1}$. Then we use the last
  estimate of the previous paragraph inductively to get
  \[
  \norm{w^s}_{H^r(\Omega_r)} \leq C_{k,\varphi_0,\ldots,\varphi_{r-1}}
  \mathscr C \max \left(\varepsilon, \mathscr S B_0^{-\frac{1}{2}
    \sqrt{2ekR \ln (\mathscr S/\varepsilon)}} \right)
  \]
  from the $L^2(\Omega)$-norm of $w^s$.
\end{proof}

\subsection*{A three spheres inequality and a chain of balls}

We state an $L^\infty$ three-balls inequality for solutions to the
Helmholtz equation. It follows from Lemma~3.5 in \cite{Rondi08} by
suitable choices of parameters. After that we prove a few lemmas and a
proposition which allows us to propagate the smallness from outside a
large ball along a straight line to near the scatterers $V$ and $V'$.

\begin{lemma}\label{3balls}
  There are positive constants $R_m, C, c_1$ such that $0<c_1<1$,
  which depend only on $k$ and satisfy the following: Let $x\in\R^n$
  and $0<4r<R_m$. If $w$ satisfies
  \[
  (\Delta + k^2) w = 0
  \]
  in $B_{4r} := B(x,4r)$, then
  \begin{equation}
    \norm{w}_{B_{2r}} \leq C (2+\sqrt{2})^{\frac{3}{2}}
    \norm{w}^{1-\beta}_{B_{4r}} \norm{w}^\beta_{B_r}
  \end{equation}
  where the norms are $L^\infty$-norms in the corresponding
  $x$-centred balls and $\beta$ is a number that satisfies
  \[
  \frac{c_1}{4} \leq \beta \leq 1 - \frac{3c_1}{4}.
  \]
\end{lemma}

\begin{proof}
  Choose $\rho_1 = r$, $\rho = 2r$, $\rho_2 = 4r$, $\tilde\rho_0 =
  R_m$ and $s = 2^{3/2}r$ in Lemma 3.5 of \cite{Rondi08}. Also choose
  $u(\cdot) = w(\cdot-x)$.
\end{proof}

\begin{lemma}\label{ballChain}
  Let $K \in \N$, $r>0$ and $B_1, \ldots, B_K$ be a chain of balls
  with the following properties:
  \begin{enumerate}
    \item $4r<R_m$, the latter defined in Lemma~\ref{3balls},
    \item the radius of each $B_k$ is $r$,
    \item the centre-to-centre distance of $B_k$ to $B_{k+1}$ is at
      most $r$.
  \end{enumerate}  
  Let $U\subset\R^n$ be open and $w\in L^\infty(U)$ satisfy the
  Helmholtz equation $(\Delta + k^2)w = 0$ there, and
  $\norm{w}_{L^\infty(U)} \leq \mathcal{T}$ which we assume to be at
  least $1$. Assume that each $B_k \subset U$ and moreover that
  $d(B_k, \partial U) \geq 3r$.
  
  Then there are finite $C\geq1$, $0<c_2<1/4$ depending only on $k$
  such that
  \[
  \norm{w}_{B_K} \leq C \mathcal{T} \norm{w}_{B_1}^{c_2^{K-1}}
  \]
  if $\norm{w}_{B_1}\leq1$, where the norms are the $L^\infty$-norms
  in the corresponding balls.
\end{lemma}

\begin{proof}
  Lemma~\ref{3balls} and the fact that $B_k$ is covered by the
  $2r$-radius ball with same centre as $B_{k-1}$ implies that
  \[
  \norm{w}_{B_k} \leq C (2+\sqrt{2})^{3/2} \mathcal{T}^{1-\beta}
  \norm{w}_{B_{k-1}}^\beta.
  \]
  Estimate $\norm{w}_{B_K}$ as above and continue telescopically to
  get
  \[
  \norm{w}_{B_K} \leq C^{1+\beta+\cdots+\beta^{K-2}}
  (2+\sqrt{2})^{\frac{3}{2}(1+\beta+\cdots+\beta^{K-2})}
  \mathcal{T}^{(1-\beta)(1+\beta+\cdots+\beta^{K-2})}
  \norm{w}_{B_1}^{\beta^{K-1}}.
  \]
  Note that $1+\cdots+\beta^{K-2} \leq 1/(1-\beta) \leq 4/(3c_1)$ and
  $\beta \geq c_1/4$. The claim follows by setting $c_2 = c_1/4$.
\end{proof}

\begin{corollary}\label{smallnessPropagation}
  Let $U\subset\R^n$ be open, $w\in L^\infty(U)$ such that
  $(\Delta+k^2)w = 0$. Let $\gamma \subset U$ be a rectifiable curve
  between two different points $x, x' \in U$ such that $B(\gamma,4r) =
  \cup_{y\in\gamma} B(y,4r) \subset U$ for some $r>0$. Assume that the
  $L^\infty$-norms satisfy $\norm{w}_{B(x,r)} \leq 1$ and that
  $\norm{w}_{U} \leq \mathcal{T}$ which is at least one.
  
  Then for any $y \in \gamma$ we have
  \[
  \norm{w}_{B(y,r)} \leq C \mathcal{T}
  \norm{w}_{B(x,r)}^{c_2^{d_\gamma(x,y)/r+1}} \leq C \mathcal{T}
  \norm{w}_{B(x,r)}^{c_2^{d_\gamma(x,x')/r+1}}
  \]
  if $4r \leq R_m$ as in Lemma~\ref{3balls}. Here $d_\gamma$ is the
  distance measured along $\gamma$.
\end{corollary}

\begin{proof}
  Denote $l = d_\gamma(x,y)$. We build a sequence of balls, each of
  radius $r$ and centres $x_1=x, x_2, x_3, \ldots, x_{\lceil l/r
    \rceil}$. Finally set $x_{\lceil l/r \rceil+1} = y$. Choose them
  so that $d_\gamma(x_{k+1},x_k)\leq r$. Hence also
  $d(x_{k+1},x_k)\leq r$. For example if $l = 2r$ we would get the
  triple $x, x_2, y$ with $2 = \lceil l/r \rceil$. For $l =
  (2+\frac{1}{2})r$ we would get the 4-tuple $x, x_2, x_3, y$ with $3
  = \lceil l/r \rceil$.  Then use the previous lemma with $B_k =
  B(x_k,r)$ and $K = \lceil l/r \rceil + 1 \leq l/r + 2$.  Since
  $\norm{w}_{B(x,r)} \leq 1$ and $c_1/4 < 1$ both estimates follow.
\end{proof}

\bigskip We are now ready to state and prove the propagation of
smallness in the context of corner scattering. Recall that $P$ and
$P'$ contain the supports of the potentials $V$, $V'$, and both are
contained in $B_R=B(0,R)$ for some fixed $R>0$. Moreover both are
convex. This is important to ensure that $B_R \setminus (P \cup P')$
is simply connected.

\begin{proposition}\label{propagateOutsideCorner}
  Let $Q \subset B_R \subset \R^n$ be a convex polytope. Let $w$ be a
  function such that $w \in L^\infty(B_{2R} \setminus Q)$ satisfies
  $(\Delta+k^2)w = 0$ in its domain, with $L^\infty$-norm at most
  $\mathcal{T}\geq1$. Let $4r \leq R_m$, the latter being from
  Lemma~\ref{3balls}, and $2r < (1-2\lambda)R$ for some positive
  $\lambda < \frac12$.
  
  Assume that $\norm{w}_{L^\infty} \leq \delta \leq 1$ in
  $B_{(2-\lambda)R} \setminus B_{(1+\lambda)R}$. Then
  \[
  \norm{w}_{L^\infty(B_{2R} \setminus B(Q, 4r))} \leq C \mathcal{T}
  \delta^{c_2^{(2+\lambda)R/r + 2}}
  \]
  where $C\geq1$ and $0<c_2<1/4$ are as in Lemma~\ref{ballChain}.
\end{proposition}

\begin{proof}
  Let $x' \in B_{2R} \setminus B(Q, 4r)$. Since $Q$ is convex there is
  a ray from $x'$ into $B_{2R} \setminus B_{(1+\lambda)R}$ that's at
  least distance $4r$ from $Q$. It can be constructed as follows:
  consider the line from $0$ to $x'$ (if $x'=0$ any line is fine). The
  point $x'$ splits it into two rays. Take one of them not touching
  the convex set $B(Q,4r)$.
  
  Cut a segment from the ray, starting at $x'$ and ending distance $r$
  outside $B_{(1+\lambda)R}$ to make sure that $\norm{w}_\infty \leq
  \delta$ in the first ball in the chain of balls we are about to use.
  This ball has radius $r$ and since $2r < (1-2\lambda)R$ it fits
  completely inside $B_{(2-\lambda)R} \setminus B_{(1+\lambda)R}$. The
  length of that segment is then at most $R+(1+\lambda)R + r$. Then
  use Corollary~\ref{smallnessPropagation}.
\end{proof}

\subsection*{Propagation of smallness into the perturbation}

The purpose of the following proposition is to estimate $u-u'$ and
$\nabla u - \nabla u'$ in Proposition~\ref{intByPartsEstim}. This is
possible because these differences are H\"older-continuous: the case
of $u-u'$ follows directly from Sobolev embedding in $\R^2$ and $\R^3$
because $V, V' \in H^s(\R^n)$ for $s<1/2$. The smoothness of the
gradient follows from elliptic regularity estimates for boundary value
problems with smooth boundary values. After all, $u-u'$ is real
analytic outside of the supports of the potentials $V$ and $V'$.

\begin{proposition}\label{prop2scatterer}
  Let $Q \subset B_R \subset \R^n$ be a convex polytope. Let $w \in
  L^\infty(B_{2R})$ be such that $w \in C^\alpha(\overline B_{3R/2})$
  with norm at most $\mathcal T \geq 1$ for some $0<\alpha<1$ and it
  satisfies $(\Delta+k^2)w=0$ in $B_{2R} \setminus Q$.

  Assume that $\abs{w(x)} \leq \delta$ in $B_{(2-\lambda)R} \setminus
  B_{(1+\lambda)R}$ for some positive $\lambda<\frac12$ and let
  $A\geq2+\lambda$. If
  \begin{equation}\label{deltaMax}
    \delta < 1\big/\exp\exp\left(\frac{4AR\abs{\ln
        c_2}/(1-\alpha)}{\min(R_m, R/2, 2(1-2\lambda)R)} \right)
  \end{equation}
  where $R_m$ is given in Lemma~\ref{3balls} then
  \begin{equation}\label{borderValueMax}
    \abs{w(x)} \leq \frac{(8AR\abs{\ln c_2}/(1-\alpha))^\alpha +
      C/c_2^2}{( \ln\abs{\ln\delta} )^\alpha} \mathcal T
  \end{equation}
  for $x\in B_{3R/2}$ satisfying $d(x, \partial Q) \leq 4AR\abs{\ln
    c_2}/((1-\alpha) \ln\abs{\ln\delta})$. Here $C$ and $c_2$ are
  given by Lemma~\ref{ballChain}.
\end{proposition}

\begin{proof}
  Choose
  \[
  r = r(\delta) = \frac{AR \abs{\ln
      c_2}}{(1-\alpha)\ln\abs{\ln\delta}}
  \]
  with $c_2$ from Lemma~\ref{ballChain}. Then $r>0$. By the upper
  bound on $\delta$ we have $4r < R_m$ and $2r < (1-2\lambda)R$ as
  required in Proposition~\ref{propagateOutsideCorner}. By that same
  proposition
  \[
  \abs{w(x')} \leq C \mathcal T \delta^{c_2^{(2+\lambda)R/r+2}}
  \]
  when $x'\in B_{2R}$, $d(x', Q) \geq 4r$.

  Let $d(x, \partial Q) \leq 4r$ now. Then there is $y\in\partial Q$
  such that $\abs{x-y}\leq 4r$. By the convexity of $Q$ there is
  $x'\in\R^n$ with $d(x',Q)=4r$ and $\abs{x'-y}=4r$. The upper bound
  on $\delta$ implies $R+4r < 3R/2$, and so $\abs{x'} \leq \abs{x'-y}
  + \abs{y} \leq 4r + R \leq 3R/2$. Thus $x' \in B_{3R/2} \setminus
  B(Q,4r)$ and $\abs{x-x'} \leq \abs{x-y}+\abs{y-x'} \leq
  8r$. Concluding, by the H\"older continuity of $w$ we have
  \[
  \abs{w(x)} \leq \norm{w}_{C^\alpha(\overline B_{3R/2})}
  \abs{x-x'}^\alpha + \abs{w(x')} \leq \mathcal{T} 8^\alpha r^\alpha +
  C \mathcal{T} \delta^{c_2^{(2+\lambda)R/r+2}}
  \]
  for $d(x,\partial Q)\leq 4r$.
  
  The choice of $r = r(\delta)$ implies that
  \[
  r^\alpha = \frac{(AR\abs{\ln c_2}/(1-\alpha))^\alpha}
  {(\ln\abs{\ln\delta})^\alpha}, \qquad \frac{(2+\lambda)R}{r} =
  -\frac{(2+\lambda)(1-\alpha)}{A} \frac{\ln\abs{\ln\delta}}{\ln c_2},
  \]
  and so
  \[
  \delta^{c_2^{(2+\lambda)R/r+2}} = e^{-\abs{\ln\delta}
    c_2^{(2+\lambda)R/r+2}} = e^{-c_2^2
    \abs{\ln\delta}^{1-(2+\lambda)(1-\alpha)/A}}.
  \]
  Now, since $2+\lambda\leq A$ and $\abs{\ln\delta}>1$, we can
  continue the above with
  \[
  \ldots \leq e^{-c_2^2 \abs{\ln\delta}^\alpha} \leq \frac{1}{ c_2^2
    \abs{\ln\delta}^\alpha } \leq \frac{1}{c_2^2
    (\ln\abs{\ln\delta})^\alpha}.
  \]
  The claim follows.
\end{proof}

\subsection*{Quantitative Rellich's theorem}

\begin{lemma}\label{HolderRegularity}
  Let $n\in\{2,3\}$ and $q\in L^\infty(B_{2R})$ be supported in $B_R$
  for some $R>0$. Let $w\in H^2(B_{2R})$ and assume that
  \[
  (\Delta + k^2(1+q))w = 0.
  \]
  Then $w \in C^{1,\frac{1}{2}}(\overline B_{3R/2})$ and there is $C =
  C(R,k,n)$ such that
  \[
  \norm{w}_{C^{1,\frac{1}{2}}} \leq C \left( 1 + \norm{q}_{L^\infty}
  \right) \norm{w}_{H^2}.
  \]
\end{lemma}

\begin{proof}
  Interior elliptic regularity in the domain where $q \equiv 0$ (e.g.
  Theorem 8.10 by Gilbarg and Trudinger \cite{GT}) implies that $w \in
  H^s(B_{7R/4} \setminus \overline{B}_{5R/4})$ and a corresponding
  norm estimate for any $s\geq0$ and in particular $s =
  (n+3)/2$. Adding Sobolev embedding gives then
  \begin{equation}\label{wHolderEstimate}
    \norm{w}_{C^{1,\frac{1}{2}}(B_{7R/4}\setminus B_{5R/4})} \leq C
    \norm{w}_{H^\frac{n+3}{2}(B_{7R/4}\setminus B_{5R/4})} \leq C
    \norm{w}_{H^2(B_{2R}\setminus B_R)}
  \end{equation}
  for some other constant $C = C(R,k,n)$. This implies that $w$ has
  boundary values in $C^{1,1/2}(\partial B_{3R/2})$, i.e. more
  precisely that there is $\varphi \in C^{1,1/2}(\R^n)$ supported in
  $B_{7R/4} \setminus B_{5R/4}$ such that $w = \varphi$ on $\partial
  B_{3R/2}$.

  Consider the Dirichlet problem for $v$
  \begin{equation} \label{Poisson}
    \Delta v = -k^2(1+q)w, \quad B_{3R/2}, \qquad v = \varphi, \quad
    \partial B_{3R/2}.
  \end{equation}
  We have $-k^2(1+q)w \in L^\infty$ and $\varphi \in
  C^{1,1/2}$. Theorem 8.34 in \cite{GT} gives unique solvability in
  the space of $C^{1,1/2}(\overline{B}_{3R/2})$-functions. However to
  conclude that $w=v$ and a fortiori $w\in C^{1,1/2}$ we need
  something more. Consider equation \eqref{Poisson} in
  $H^1(B_{3R/2})$. In this space both $v$ and $w$ are solutions and
  they satisfy
  \[
  \Delta (v-w) = 0, \quad B_{3R/2}, \qquad v-w = 0, \quad \partial
  B_{3R/2}.
  \]
  By the $H^1$-maximum principle $v=w$ in $H^1$. Hence $w \in
  C^{1,1/2}$.

  Finally, Theorem 8.33 in \cite{GT} gives an estimate for $\norm{v}$
  in $C^{1,1/2}(\overline B_{3R/2})$ based on the boundary and source
  data. Using that, the Sobolev embedding of $H^2\hookrightarrow
  L^\infty$ in two and three dimensions, and \eqref{wHolderEstimate}
  gives
  \[
  \norm{w}_{C^{1,\frac{1}{2}}(\overline B_{3R/2})} \leq C \left(
  \norm{w}_{H^1(B_{2R})} +\norm{-k^2(1+q)w}_{L^\infty(B_{2R})}
  \right)
  \]
  for some constant $C = C(R,n)$ and the claim follows.
\end{proof}

\begin{proposition}\label{FF2bndry}
  Let $R>1$, $n\in\{2,3\}$ and $k>0$. Let $u^i \in H^2_{loc}(\R^n)$ be
  an incident wave, $(\Delta+k^2)u^i=0$, with
  $\norm{u^i}_{H^2(B_{2R})} \leq \mathcal I$.
  
  Let $P, P' \subset B_R$ be open convex polytopes, and $\varphi,
  \varphi' \in L^\infty(\R^n)$.  Let $V = \chi_P \varphi$ and $V' =
  \chi_{P'} \varphi'$ be two potentials with $\norm{V}_\infty,
  \norm{V'}_\infty \leq \mathcal M$. Also, let $u,u' \in
  H^2_{loc}(\R^n)$ be total waves satisfying
  \[
  (\Delta+k^2(1+V))u=(\Delta+k^2(1+V'))u'=0
  \]
  and whose scattered waves $u^s = u-u^i$, $u'^s = u'-u^i$ satisfy the
  Sommerfeld radiation condition. Let $u^s_\infty, u'^s_\infty :
  \mathbb{S}^{n-1} \to \C$ be their far-field patterns.
  
  Assume that $\norm{u^s},\norm{{u'}^s}\leq \mathcal S$ in
  $H^2(B_{2R})$ and $\mathcal S \geq 1$. Then there is $\varepsilon_m
  = \varepsilon_m(\mathcal S, k, R)>0$ such that if
  \[
  \norm{u^s_\infty - u'^s_\infty}_{L^2(\mathbb S^{n-1})} \leq
  \varepsilon_m
  \]
  and $Q$ is the convex hull of $P$ and $P'$ then $u-u', \nabla u - 
  \nabla u'$ are continuous in $B_{R}$ and
  \[
  \sup_{\partial Q} \big(\abs{u-u'} + \abs{\nabla u - \nabla u'}\big)
  \leq \mathcal C \left(\ln \ln(\mathcal S\norm{u^s_\infty -
    u'^s_\infty}_{L^2(\mathbb S^{n-1})}^{-1} ) \right)^{-1/2}
  \]
  for some $\mathcal C = C(k,R)(1+\mathcal M)(\mathcal I+\mathcal S)$.
\end{proposition}

\begin{proof}
  Firstly, propagate smallness from the far-field to the near-field by
  using Corollary~\ref{usedFF2NF}. Let $w^s$ in that proposition be
  $u-u' = u^s - u'^s$ and denote $\varepsilon = \norm{u^s_\infty -
    u'^s_\infty}_{L^2(\mathbb S^{n-1})}$. Note also that
  $\norm{w^s}_{H^2(B_{2R})} \leq 2 \mathcal S$ then. Choose the
  annulus $A = B_{(2-\lambda)R} \setminus \overline B_{(1+\lambda)R}$
  for some positive $\lambda < \frac12$. Corollary~\ref{usedFF2NF}
  implies that $w^s\in H^r(A)$ for any $r\in\N$. Moreover in two and
  three dimensions Sobolev embedding implies that
  $H^2(A)\hookrightarrow L^\infty(A)$. Hence the estimate given by the
  corollary becomes
  \[
  \norm{w^s}_{L^\infty(A)}, \norm{\nabla w^s}_{L^\infty(A)} \leq C
  \max \left( \varepsilon, \mathcal S e^{-c \sqrt{\ln(\mathcal S/
      \varepsilon)}} \right)
  \]
  for $C>1, c>0$ depending on $k,R,\lambda$ (we estimated
  $\ln(2\mathcal S/\varepsilon) \geq \ln(\mathcal
  S/\varepsilon)$). Our first requirement on $\varepsilon_m$ is that
  the maximum picks the number on the right side. This happens if
  $\varepsilon \leq \mathcal S e^{-c^2}$ so let us require
  $\varepsilon_m \leq \mathcal S e^{-c^2}$.
  
  The second step is to use the propagation of smallness by
  Proposition~\ref{prop2scatterer} for $w = w^s$ and also for $w =
  \partial_j w^s$, $j=1,\ldots,n$. By Lemma~\ref{HolderRegularity} we
  have
  \[
  \norm{u}_{C^{1,1/2}} \leq C (1+\mathcal M)(\mathcal I +\mathcal S)
  \]
  in $\overline B_{3R/2}$ and similarly for $u'$. So $w^s, \partial_j
  w^s \in C^{1/2}$ for each $j$. Thus the smoothness requirements of
  Proposition~\ref{prop2scatterer} are satisfied for each choice of
  $w$. Also $C=C(k,R)$.  Set
  \[
  \delta = C \mathcal S e^{-c\sqrt{\ln(\mathcal S/\varepsilon)}}.
  \]
  We get a second upper bound on $\varepsilon_m$ by requiring that
  $\delta$ satisfies \eqref{deltaMax}. The right-hand side in that
  inequality depends only on $A = A(\lambda,R)$, $k$ and $R$, so this
  second, updated, upper bound for $\varepsilon_m$ still only depends
  on $\lambda,k,R$. Now Proposition~\ref{prop2scatterer} implies
  \[
  \abs{w(x)} \leq C(1+\mathcal M) (\mathcal I + \mathcal S)
  (\ln\abs{\ln\delta})^{-1/2}
  \]
  with $C=C(\lambda,k,R)$. The choice of $\delta$ implies
  \[
  \abs{\ln\delta} = c \sqrt{\ln(\mathcal S \varepsilon^{-1})} -
  \ln(C\mathcal S) \geq \frac{c}{2} \sqrt{\ln(\mathcal S
    \varepsilon^{-1})} \geq \big(\ln(\mathcal S
  \varepsilon^{-1})\big)^{1/4}
  \]
  if $\varepsilon_m$ is small enough (and again $c, C$ depend only on
  $k,\lambda,R$). Thus
  \[
  (\ln\abs{\ln\delta})^{-1/2} \leq \Big(\ln \big( \ln(\mathcal S
  \varepsilon^{-1})\big)^{1/4} \Big)^{-1/2} = 2 \big(\ln \ln (\mathcal
  S \varepsilon^{-1}) \big)^{-1/2}
  \]
  and the claim follows after choosing $\lambda$ as a function of $R$
  for example.
\end{proof}

\section{From boundary to inside} \label{sect:orthogonality}

We deal with particulars related to corner scattering in this section.
More precisely, we prove the fundamental orthogonality identity which
is the foundation upon which past results \cite{BPS, PSV, HSV} were
built on. Since we are proving stability instead of uniqueness we have
an extra boundary term here to deal with.  Moreover, for future
convenience, we do not assume that $u^i(x_c) \neq 0$ in
Proposition~\ref{intByPartsEstim}. This does not complicate the
argument by much.

\begin{proposition}\label{intByParts}
  Let $Q_h \subset \R^n$ be a bounded Lipschitz domain, $V\in
  L^\infty(Q_h)$, $k>0$ and $u^i,u,u_0 \in H^2(Q_h)$ satisfy
  \begin{align*}
  &(\Delta+k^2)u^i=0,\\ &(\Delta+k^2(1+V))u=0,\\ &(\Delta+k^2(1+V))u_0=0
  \end{align*}  
  in $Q_h$. Then
  \begin{equation}
    k^2 \int_{Q_h} V\,u_0\,u^i\,dx = \int_{\partial Q_h} \big(u_0\,
    \partial_\nu (u^i-u) - (u^i-u)\, \partial_\nu u_0\big) d\sigma.
  \end{equation}
\end{proposition}

\begin{proof}
  Use Green's formula after noting that
  \[
  k^2 \int_{Q_h} V\,u_0\,u^i\,dx = \int_{Q_h} u_0 \big( \Delta +
  k^2(1+V)\big) (u^i-u)\,dx.
  \]
\end{proof}

We consider only incident waves that do not vanish anywhere in this
paper. This means that in the following corollary we would always have
$P_N$ a constant and $N=0$. The corollary is stated so that it applies
also to the more general case where the incident wave can vanish up to
a finite order $N$ at $x_c$. This is for the convenience of future
papers on the topic and also since the proof is not substantially more
difficult in this case.

\begin{proposition} \label{intByPartsEstim}
  Let $\mathfrak P, \mathfrak Q \subset \R^n$ be open polyhedral cones
  with vertex $x_c$ such that $\mathfrak P \subset \mathfrak Q$ and
  their boundaries are a subset of the union of at most $\mathcal V$
  hyperplanes of codimention 1. Let $P_h = \mathfrak P \cap B(x_c,h)$
  and $Q_h = \mathfrak Q \cap B(x_c,h)$ for $0<h\leq1$.
  
  Let $k>0$ and $V,V' \in L^\infty(\R^n)$ be supported in $B_R\supset
  Q_h$ for some $R>1$. Assume that $V = \chi_{\mathfrak P} \varphi$
  and $V'=0$ in $Q_h$ for some measurable function $\varphi:P_h \to
  \C$. Let $u, u', u_0 \in H^2(B_{2R})$ satisfy
  \[
  (\Delta+k^2(1+V))u=(\Delta+k^2(1+V))u_0=0, \quad
  (\Delta+k^2(1+V'))u'=0.
  \]
  If we have
  functions $P_N$, $\varphi_\alpha$, $u'_{N+1}$, $\psi$ and a complex
  vector $\rho\in\C^n$ such that
  \begin{align*}
    &\varphi(x) = \varphi(x_c) + \varphi_\alpha(x),\\ &u'(x) =
    P_N(x-x_c) + u'_{N+1}(x),\\ &u_0(x) =
    e^{\rho\cdot(x-x_c)}(1+\psi(x)),
  \end{align*}  
  in $P_h$, then
  \begin{align}
    & \varphi(x_c)\int_{\mathfrak P} e^{\rho\cdot(x-x_c)} P_N(x-x_c)
    dx = \varphi(x_c)\int_{\mathfrak P \setminus P_h}
    e^{\rho\cdot(x-x_c)} P_N(x-x_c) dx \notag\\ &\qquad - \int_{P_h}
    e^{\rho\cdot(x-x_c)} \varphi_\alpha(x) P_N(x-x_c) dx - \int_{P_h}
    e^{\rho\cdot(x-x_c)} V(x) u'_{N+1}(x) dx \notag\\ &\qquad -
    \int_{P_h} e^{\rho\cdot(x-x_c)} V(x) u'(x) \psi(x) dx +
    \frac{1}{k^2} \int_{\partial Q_h} \big( u_0 \partial_\nu (u-u') -
    (u-u') \partial_\nu u_0 \big) d\sigma
  \label{finalOrthEquality}.
  \end{align}

  Assume moreover that $\psi\in L^p$ in $Q_h$, $p>1$, and that
  \begin{enumerate}
    \item $\abs{\rho}\geq1$ and $\Re\rho\cdot(x-x_c) \leq - \delta_0
      \abs{x-x_c}\abs{\Re\rho}$ for some $\delta_0>0$ and any $x\in
      Q_h$,
    \item $\abs{\varphi_\alpha(x)} \leq \mathcal M\abs{x-x_c}^\alpha$,
      $\abs{V(x)} \leq \mathcal M$ for $x\in P_h$, and some $\alpha>0$
    \item $\abs{u'(x)} \leq \mathcal F\abs{x-x_c}^N$ for $x\in P_h$,
    \item $\abs{P_N(x-x_c)} \leq \mathcal P\abs{x-x_c}^N$ for $x\in
      P_h$,
    \item $\abs{u'_{N+1}(x)} \leq \mathcal R \abs{x-x_c}^{N+1}$ for
      $x\in P_h$,
  \end{enumerate}
  with $0\leq N \leq \mathcal N$ then we have the norm estimate
  \begin{align}
  &C\abs{\varphi(x_c) \int_{\mathfrak P} e^{\rho\cdot(x-x_c)}
      P_N(x-x_c) dx} \leq \abs{\Re\rho}^{-N-n}
    e^{-\delta_0\abs{\Re\rho}h/2} \notag\\ &\qquad\quad +
    \abs{\Re\rho}^{-N-n-\min(1,\alpha)} + \abs{\Re\rho}^{-N-n/p'}
    \norm{\psi}_{L^p(P_h)} \notag\\ &\qquad \quad + h^{(n-1)/2}
    \abs{\rho} (1+\norm{\psi}_{H^2(B_{2R})}) \sup_{\partial \mathfrak
      Q\cap B(x_c,h)} \{\abs{u-u'}, \abs{\nabla u-\nabla u'}\}
    \notag\\ &\qquad \quad + h^{n/2-1} e^{-\delta_0\abs{\Re\rho}h}
    \abs{\rho} (1+\norm{\psi}_{H^2(B_{2R})})
    (\norm{u}_{H^2(B_{2R})}+\norm{u'}_{H^2(B_{2R})})\label{integralEstims}
  \end{align}
  where $1/p+1/p'=1$ and $C>0$ depends on all the a-priori parameters
  $\mathcal V, k, \mathcal P, \mathcal M, \mathcal N, \mathcal R,
  \mathcal F, \alpha, \delta_0, n, p$.
\end{proposition}

\begin{proof}
  The integral identity is a direct calculation using
  Proposition~\ref{intByParts} with $Q_h$ and $u^i=u'$, and then
  noting that $V = 0$ on $Q_h \setminus P_h$. For the others we use
  the incomplete gamma functions $\gamma, \Gamma: \R_+\times\R_+ \to
  \R$
  \[
  \gamma(s,x) = \int_0^x e^{-t} t^{s-1} dt, \qquad \Gamma(s,x) =
  \int_x^\infty e^{-t} t^{s-1} dt
  \]
  which satisfy $\gamma(s,x) \leq \Gamma(s) \leq \lceil s-1 \rceil!$
  and $\Gamma(s,x) \leq 2^s \Gamma(s) e^{-x/2}$, where $\Gamma(s)$
  represents the ordinary, complete, gamma function. The latter
  estimate follows from splitting $e^{-t} \leq e^{-t/2} e^{-x/2}$ in
  the integral, expanding the integration limits to ${({0,\infty})}$
  and switching to the integration variable $t' = t/2$. By a radial
  change of coordinates the first integral on the right has the upper
  bound
  \begin{align*}
    &\abs{\int_{\mathfrak P\setminus P_h}
      e^{\rho\cdot(x-x_c)}P_N(x-x_c) dx} \leq \int_{\mathfrak
      P\setminus P_h} e^{-\delta_0\abs{\Re\rho}\abs{x-x_c}} \mathcal P
    \abs{x-x_c}^N dx \\ &\qquad \leq \mathcal P \sigma(\mathbb
    S^{n-1}) \int_h^\infty e^{-\delta_0\abs{\Re\rho} r} r^{N+n-1} dr
    \\ &\qquad \leq \left(\frac{2}{\delta_0}\right)^{N+n} (N+n)!\,
    \mathcal P\, \sigma(\mathbb S^{n-1}) \abs{\Re\rho}^{-N-n}
    e^{-\delta_0 \abs{\Re\rho} h /2} \\ &\qquad \leq C_{\delta_0,
      \mathcal N, n, \mathcal P} \abs{\Re\rho}^{-N-n} e^{-\delta_0
      \abs{\Re\rho} h /2}
  \end{align*}
  for the first integral on the right.

  For the integral inside $P_h$ note
  \begin{align*}
    &\int_{P_h} e^{\Re\rho\cdot(x-x_c)q'} \abs{x-x_c}^{Bq'} dx =
    \int_{\mathbb S^{n-1} \cap (P_h-x_c)} \int_0^h e^{-\delta_0 q'
      \abs{\Re\rho} r} r^{Bq'+n-1} dr d\sigma(\theta) \\ &\qquad \leq
    \sigma(\mathbb S^{n-1}) \int_0^{\delta_0 q' \abs{\Re\rho} h}
    e^{-r'} {r'}^{Bq'+n-1} \frac{dr'}{(\delta_0 q'
      \abs{\Re\rho})^{Bq'+n}} \\ &\qquad = \sigma(\mathbb S^{n-1})
    \gamma( Bq'+n, \delta_0 q' \abs{\Re\rho} h ) (\delta_0 q'
    \abs{\Re\rho})^{-Bq'-n}.
  \end{align*}
  Use this to prove the following estimate, each of which shall be applied to
  the next three integrals in \eqref{finalOrthEquality}. Let $f,g$ be
  functions such that $\abs{f(x)} \leq A\abs{x-x_c}^B$ with
  $A\leq\mathcal A$, $B\leq\mathcal B$, and $g \in L^q$. Then
  \begin{align*}
    &\abs{\int_{P_h} e^{\rho\cdot(x-x_c)} f(x)\, g(x)\, dx} \leq A
    \left(\int_{P_h} e^{\Re\rho\cdot(x-x_c)q'} \abs{x-x_c}^{Bq'} dx
    \right)^{1/q'} \norm{g}_{L^q(P_h)} \\ &\quad \leq A \left(\frac{
      \sigma(\mathbb S^{n-1}) \gamma(Bq'+n, \delta_0q'\abs{\Re\rho}h)
    }{(\delta_0 q' \abs{\Re\rho})^{Bq'+n}}\right)^{1/q'}
    \norm{g}_{L^q(P_h)} \\ &\qquad \leq A \left( \frac{\sigma(\mathbb
      S^{n-1}) \lceil Bq'+n \rceil! }{ (\delta_0 q'
      \abs{\Re\rho})^{Bq'+n} } \right)^{1/q'} \norm{g}_{L^q(P_h)}
    \\ &\qquad \leq C_{\mathcal A, \mathcal B, n, \delta_0, q}
    \abs{\Re\rho}^{-B-n/q'} \norm{g}_{L^q(P_h)}
  \end{align*}
  where $1/q+1/q'=1$. Choosing
  \begin{itemize}
    \item $q=\infty$, $\mathcal A = \mathcal P \mathcal M$,
      $B=N+\alpha \leq \mathcal N+\alpha$,
    \item $q=\infty$, $\mathcal A = \mathcal M \mathcal R$, $B=N+1
      \leq \mathcal N+1$, and
    \item $q=p$, $\mathcal A = \mathcal M \mathcal F$, $B=N \leq
      \mathcal N$
  \end{itemize}
  gives the three estimates
  \begin{align*}
    &\abs{\int_{P_h} e^{\rho\cdot(x-x_c)} \varphi_\alpha(x)\,
      P_N(x-x_c)\, dx} \leq C_{\mathcal P, \mathcal M, \mathcal N,
      \alpha, n, \delta_0}
    \abs{\Re\rho}^{-N-n-\alpha},\\ &\abs{\int_{P_h}
      e^{\rho\cdot(x-x_c)} V(x)\, u'_{N+1}(x)\, dx} \leq C_{\mathcal
      M, \mathcal R, \mathcal N, n, \delta_0}
    \abs{\Re\rho}^{-N-n-1},\\ &\abs{\int_{P_h} e^{\rho\cdot(x-x_c)}
      V(x)\, u'(x)\, \psi(x)\, dx} \leq C_{\mathcal M, \mathcal F,
      \mathcal N, n, \delta_0, p} \abs{\Re\rho}^{-N-n/p'}
    \norm{\psi}_{L^p(P_h)}.
  \end{align*}

  \medskip
  Only the boundary integral is left in \eqref{finalOrthEquality}. Let
  us split the boundary into two pieces: $\partial Q_h = (\partial
  \mathfrak Q \cap B(x_c,h)) \cup (\mathfrak Q\cap S(x_c,h))$.  For
  the first piece use the Cauchy-Schwartz inequality which gives
  \begin{align*}
    &\abs{\int_{\partial \mathfrak Q\cap B(x_c,h)} \big( u_0\,
      \partial_\nu (u-u') - (u-u')\, \partial_\nu u_0\big)\,
      d\sigma(x)} \leq \sqrt{\sigma(\partial \mathfrak Q\cap
      B(x_c,h))} \, \cdot \\ &\qquad \cdot \big((1 + \abs{\rho})(1 +
    \norm{\psi}_{L^2(\partial \mathfrak Q\cap B(x_c,h))}) +
    \norm{\partial_\nu \psi}_{L^2(\partial \mathfrak Q\cap B(x_c,h))}
    \big) \norm{u-u'}_{NF}
  \end{align*}
  where $\norm{f}_{NF}$ denotes the maximum of $\abs{f}$ and
  $\abs{\nabla f}$ on $\partial \mathfrak Q\cap B(x_c,h)$. This
  estimate uses $\abs{\exp(\rho\cdot(x-x_c))} \leq 1$ in $\mathfrak
  Q$.

  Both $\norm{\psi}_2$ and $\norm{\partial_\nu \psi}_2$ can be
  estimated by $C_{\mathcal V,n} \norm{\psi}_{H^2(B_{2R})}$ in the set
  $\partial \mathfrak Q\cap B(x_c,h)$ since $h\leq1$ and so
  $B(x_c,h)\subset B_{2R}$. The constant depends on $\mathcal V$
  instead of $\mathfrak Q$ because
  \[
  \partial \mathfrak Q \subset \bigcap_{j=1}^{\mathcal V} \partial H_j
  \cap B(x_c,1)
  \]
  for some half-spaces $H_j$ that pass through $x_c$. The trace norm
  is identical in each of the sets $H_j \cap B(x_c,1)$. By an easier
  argument we see that $\sqrt{\sigma(\partial \mathfrak Q\cap
    B(x_c,h))} \leq C_{\mathcal V,n} h^{(n-1)/2}$ and the estimate for
  the first part of the boundary term in \eqref{finalOrthEquality}
  follows.

  For estimating the last integral, the one over $\mathfrak Q \cap
  S(x_c,h)$, the Cauchy-Schwartz inequality gives
  \begin{align*}
    &\abs{ \int_{\mathfrak Q\cap S(x_c,h)} \big( u_0\, \partial_\nu
      (u-u') - (u-u')\, \partial_\nu u_0\big)\, d\sigma(x)}
    \\ &\qquad \leq \sqrt{\sigma(\mathfrak Q \cap S(x_c,h))}
    \norm{u-u'}_{C^1(\overline B(x_c,h))} e^{-\delta_0\abs{\Re\rho}h}
    \cdot \\ &\qquad\quad \cdot \big( (1+\abs{\rho})
    (1+\norm{\psi}_{L^2(\mathfrak Q\cap S(x_c,h))}) +
    \norm{\partial_\nu \psi}_{L^2(\mathfrak Q\cap S(x_c,h))} \big).
  \end{align*}
  We can estimate by $C^1$-norm by Lemma~\ref{HolderRegularity} which
  gives $\norm{u-u'}_{C^{1,1/2}}\leq C (1+\mathcal M) (\norm{u}_2
  +\norm{u'}_2)$ where the $\norm{\cdot}_2$-norm is the
  $H^2(B_{2R})$-norm.

  For estimating $\psi$ let us consider how the trace-norm depends on
  $h$ when the trace-operator maps $H^1(B(x_c,h)) \to
  L^2(S(x_c,h))$. We do this by scaling the variables, for example by
  having $g(y) = f(h(y-x_c)+x_c)$ and $f(x) = g((x-x_c)/h+x_c)$. Now
  \begin{align*}
    &\norm{f}_{L^2(S(x_c,h))} = h^{\frac{n-1}{2}}
    \norm{g}_{L^2(S(x_c,1))} \leq C h^{\frac{n-1}{2}}
    \norm{g}_{H^1(B(x_c,1))} \\ &\qquad \leq C h^{-\frac{1}{2}} (1+h)
    \norm{f}_{H^1(B(x_c,h))} \leq C h^{-\frac{1}{2}}
    \norm{f}_{H^1(B(x_c,h))}
  \end{align*}
  because of $h\leq1$. Hence we see that $\norm{\psi}_2$ and
  $\norm{\partial_\nu \psi}_2$ can be estimated by $C_n h^{-1/2}
  \norm{\psi}_{H^2(B_{2R})}$ in $L^2(\mathfrak Q\cap
  S(x_c,h))$. However note that
  \[
  \sqrt{\sigma(\mathfrak Q \cap S(x_c,h))} \leq C h^{(n-1)/2}
  \]
  so the final estimate \eqref{integralEstims} follows.
\end{proof}

To prove the final stability results, we need a lower bound on the
left-hand side of \eqref{integralEstims}. This is nontrivial. In
previous papers \cite{BPS}, \cite{PSV} it is shown that the left-hand
side does not vanish. We do need a quantitative version, for example
of the form: given a polynomial $P_N$ satisfying some a-priori
conditions, the left hand side is greater than $C$ which does not
depend on $P_N$. This turns out to require a too fine analysis in the
context of support probing. However we can avoid this because we
assumed that $u'(x_c) \neq 0$, which implies that $P_N(x) \equiv
u'(x_c)$ is constant.

\begin{lemma}\label{lowerBound}
  Let $n\in\{2,3\}$, $0<2\alpha_m<2\alpha_M<2\alpha'<\pi$ and
  $k>0$. For $\mathcal Q, \mathfrak P \subset \R^n$ we say $(\mathcal
  Q, \mathfrak P) \in \mathscr G(\alpha_m, \alpha_M, \alpha', n)$ if
  the following are satisfied
  \begin{enumerate}
    \item $\mathcal Q$ is an open spherical cone,
    \item $\mathfrak P$ is an open convex polyhedral cone,
    \item $\mathcal Q$ and $\mathfrak P$ have a common vertex $x_c \in
      \R^n$,
    \item $\mathfrak P \subset \mathcal Q$,
    \item $\mathcal Q$ has opening angle at most $2\alpha'$,
    \item in 2D $\mathfrak P$ has opening angle in ${]{2\alpha_m,
      2\alpha_M}[}$,
    \item in 3D $\mathfrak P$ can be transformed to ${]{0,\infty}[}^3$
      by a rigid motion.
  \end{enumerate}

  If $(\mathcal Q, \mathfrak P) \in \mathscr G(\alpha_m, \alpha_M,
  \alpha', n)$, then there is $\tau_0 = k\,C(\alpha_m, \alpha_M,
  \alpha', n) > 0$, and $c = c(\alpha_m, \alpha_M, n)>0$ with the
  following properties. There is a curve
  $\tau\mapsto\rho(\tau)\in\C^n$ (which depends on $\mathcal Q$)
  satisfying $\rho(\tau)\cdot\rho(\tau)+k^2=0$,
  $\tau=\abs{\Re\rho(\tau)}$,
  \[
  \Re\rho(\tau)\cdot(x-x_c) \leq -\cos\alpha' \abs{\Re\rho(\tau)}
  \abs{x-x_c}
  \]
  for all $x\in \mathcal Q$ and such that if $\tau\geq\tau_0$ then
  \[
  \abs{\int_{\mathfrak P} e^{\rho(\tau)\cdot(x-x_c)} dx} \geq c
  \tau^{-n}.
  \]
\end{lemma}

\begin{proof}
  We start by proving the claim for $\zeta\cdot\zeta=0$ instead of
  $\rho\cdot\rho+k^2=0$. Consider the cases $n=2$ and $n=3$
  separately. Let $(\mathcal Q, \mathfrak P) \in \mathscr G(\alpha_m,
  \alpha_M, \alpha', 2)$. Then there is a rigid motion $M_{\mathfrak
    P}$ and $\alpha \in {[{2\alpha_m, 2\alpha_M}]}$ such that
  $M_{\mathfrak P}$ takes $\mathfrak P$ to $\{ x \in \R^2 \mid x_2 >0,
  x_1 > a x_2\}$ where $a = 1/\tan \alpha$. We have $M_{\mathfrak P} x
  = R_{\mathfrak P}(x-x_c)$ for some rotation $R_{\mathfrak
    P}$. Denote $\xi = R_{\mathfrak P} \zeta$. Then
  \[
  \int_{\mathfrak P} e^{\zeta\cdot(x-x_c)}dx = \int_0^\infty
  \int_{ay_2}^\infty e^{\xi\cdot y} dy_1 dy_2 = \frac{1}{\xi_1 (\xi_2
    + a\xi_1)}
  \]
  if $\Re \xi_1 < 0$ and $\Re (\xi_2 + a \xi_1) < 0$. If
  $\zeta\cdot\zeta=0$ and $\abs{\Re\zeta}=1$ then the same is true for
  its rotated version $\xi$ and so $\abs{\xi_1}=\abs{\xi_2}=1$. This
  implies $\abs{\xi_2/\xi_1+a} \leq 1+\abs{a}$. Thus
  \[
  \abs{\int_{\mathfrak P} e^{\zeta\cdot(x-x_c)}dx} \geq
  \frac{1}{1+\abs{a}} > 0
  \]
  because $\abs{a}$ can be estimated above by $1/\min
  \abs{\tan\alpha}$, where the minimum is taken over $2\alpha_m \leq
  \alpha \leq 2\alpha_M$, and the limits are away from $0$ and $\pi$.
  
  The conditions $\Re\xi_1<0$ and $\Re(\xi_2 + a\xi_1)<0$ are implied
  at once if
  \[
  \Re\zeta \cdot (x-x_c) \leq - \cos \alpha' \abs{x-x_c}
  \]
  for all $x\in \mathcal Q$ as this means that the map $x \mapsto
  \exp(\Re\zeta\cdot(x-x_c))$ is exponentially decreasing in $\mathcal
  Q$, and a fortiori in $\mathfrak P$. We can now build $\zeta$. Let
  $-\Re\zeta$ be the unit vector on the central axis of $\mathcal Q$
  to make the above inequality valid. Next choose $\Im\zeta$ such that
  $\Im\zeta \perp \Re\zeta$, $\abs{\Im\zeta} = 1 =
  \abs{\Re\zeta}$. This implies $\zeta\cdot\zeta = 0$.

  \smallskip
  Consider the 3D case now. Let $(\mathcal Q, \mathfrak P) \in
  \mathscr G(\alpha_m, \alpha_M, \alpha', 3)$. Then there is a rigid
  motion $M_{\mathfrak P}$ bringing $\mathfrak P$ to
  ${]{0,\infty}[}^3$. We have $M_{\mathfrak P}x = R_{\mathfrak
    P}(x-x_c)$ for some rotation $R_{\mathfrak P}$. Denote again $\xi
  = R_{\mathfrak P}\zeta$. Then
  \[
  \int_{\mathfrak P} e^{\zeta\cdot(x-x_c)}dx = \int_{{]{0,\infty}[}^3}
  e^{\xi\cdot y}dy = \frac{-1}{\xi_1\xi_2\xi_3}
  \]
  as long as $\xi_j < 0$ for all $j$. As before, $\zeta\cdot\zeta=0$
  and $\abs{\Re\zeta}=1$ imply $\abs{\xi}\leq\sqrt{2}$ and the lower
  bound of $2^{-3/2}$ for the integral. The conditions $\xi_j<0$
  follow from
  \[
  \Re\zeta\cdot(x-x_c) \leq -\cos\alpha'\abs{x-x_c}
  \]
  in $\mathcal Q$. The choice of $\zeta$ is made as in the 2D case.
  
  \smallskip
  To recap, in both 2D and 3D, for any $(\mathcal Q, \mathfrak P) \in
  \mathscr G(\alpha_m, \alpha_M, \alpha', n)$ we found $\zeta\in\C^n$
  satisfying $\zeta\cdot\zeta=0$, $\abs{\zeta}=1$,
  $\Re\zeta\cdot(x-x_c) \leq -\cos\alpha' \abs{x-x_c}$ for all $x\in
  \mathcal Q$ with $x_c$ the vertex, and finally
  \[
  \abs{\int_{\mathfrak P} e^{\zeta\cdot(x-x_c)}dx} \geq
  2C_{\alpha_m,\alpha_M,n}>0.
  \]
  Let us build the curve $\rho(\tau)$ next. Set
  \[
  \rho(\tau) = \tau \Re\zeta + i \sqrt{\tau^2 + k^2} \Im\zeta.
  \]
  Is is easy to see that $\rho(\tau)/\tau \to \zeta$ as
  $\tau\to\infty$, and even easier to see that $\Re\rho(\tau) \cdot
  (x-x_c) \leq -\cos\alpha' \abs{\Re\rho(\tau)} \abs{x-x_c}$ for $x\in
  \mathcal Q$.  Write $\mathscr L(\zeta) = \int_{\mathfrak P}
  \exp(\zeta\cdot(x-x_c))dx$ to conserve space. We quantify how far
  $\mathscr L(\rho(\tau)/\tau)$ is from $\mathscr L(\zeta)$ next.
  Ideally we want an estimate that does not depend on $\mathcal Q$ or
  $\mathfrak P$.
  
  If we set $f(r) = \exp((\Re\zeta + i r \Im\zeta)\cdot(x-x_c))$ then
  $f(1) = \exp(\zeta\cdot(x-x_c))$ and $f(\sqrt{1+k^2/\tau^2}) =
  \exp(\rho(\tau)/\tau \cdot (x-x_c))$. By the mean value theorem
  \[
  \abs{ f(1) - f\left(\sqrt{1+\tfrac{k^2}{\tau^2}}\right) } \leq
  \sup_{1<r<\sqrt{1+\frac{k^2}{\tau^2}}} \abs{f'(r)} \abs{
    \sqrt{1+\frac{k^2}{\tau^2}} - 1}.
  \]
  Note that $\sqrt{1+k^2/\tau^2} - 1 \leq k/\tau$. Also $f'(r) = i
  \Im\zeta\cdot(x-x_c) f(r)$ and because $\abs{\Im\zeta} =
  \abs{\Re\zeta} = 1$ we have $\abs{f'(r)} \leq \abs{x-x_c}
  \exp(-\cos\alpha'\abs{x-x_c})$. Hence
  \[
  \abs{ f(1) - f\left(\sqrt{1+\tfrac{k^2}{\tau^2}}\right) } \leq
  \frac{k}{\tau} \abs{x-x_c} e^{-\cos\alpha'\abs{x-x_c}}.
  \]
  We see finally that
  \begin{align*}
    &\abs{ \mathscr L(\zeta) - \mathscr
      L\left(\frac{\rho(\tau)}{\tau}\right) } = \abs{ \int_{\mathfrak
        P} \big( f(1) - f(\sqrt{1+k^2/\tau^2}) \big) dx } \\ &\qquad
    \leq \frac{k}{\tau} \int_{\mathfrak P} e^{-\cos\alpha'
      \abs{x-x_c}} \abs{x-x_c} dx \\ &\qquad \leq \sigma(\mathfrak P
    \cap \mathbb S^{n-1}) \frac{k}{\tau} \int_0^\infty e^{-\cos\alpha'
      r} r^{1+n-1} dr \leq C_{\alpha',n} k \tau^{-1}
  \end{align*}
  because we can estimate $\sigma(\mathfrak P \cap \mathbb S^{n-1})
  \leq \sigma(\mathbb S^{n-1})$, and $\cos\alpha' > 0$ since $\alpha'
  < \pi/2$.

  Now, it is easily seen that $\mathscr L(\rho(\tau)/\tau) = \tau^n
  \mathscr L(\rho(\tau))$. Recall that our choice of $\zeta$ implies
  that $\abs{\mathscr L(\zeta)} \geq 2C_{\alpha_m, \alpha_M, n}$. By
  the triangle inequality
  \begin{align*}
    &\abs{ \tau^n \mathscr L(\rho(\tau)) } = \abs{ \mathscr
      L\left(\frac{\rho(\tau)}{\tau}\right) } \geq \abs{\mathscr
      L(\zeta)} - \abs{ \mathscr L(\zeta) - \mathscr
      L\left(\frac{\rho(\tau)}{\tau}\right)} \\ &\qquad >
    2C_{\alpha_m, \alpha_M, n} - C_{\alpha',n} k \tau^{-1} \geq
    C_{\alpha_m, \alpha_M, n} > 0
  \end{align*}
  if $\tau \geq C_{\alpha',n} k / C_{\alpha_m,\alpha_M,n}$ which is
  finite and depends only on the a-priori parameters.
\end{proof}

\section{Complex geometrical optics solution} \label{sect:cgo}

The construction of the CGO solutions for corner scattering was first
shown in \cite{BPS} and \cite{PSV}. We do the analysis more precisely
and keep track of what parameters the various bounds depend on. This
involves defining a ``norm'' for polyhedral regions. We start by
solving the Faddeev equation, then prove estimates for potentials
supported on polytopes and finally build the complex geometrical
optics solutions.

\begin{lemma}\label{FaddeevSol}
  Let $s\geq0$ and $1<r<2$ such that $1/r + 1/r'=1$ and
  \[
  \frac{2}{n+1} \leq \frac{1}{r} - \frac{1}{r'} < \frac{2}{n}.
  \]
  Let $q$ be a measurable function such that the pointwise multiplier
  operator $m_q$ maps $H^s_{r'}(\R^n) \to H^s_r(\R^n)$, and let $f \in
  H^s_r(\R^n)$.
  
  Let $I_0 = (2 M \norm{ m_q }_{H^s_{r'} \to H^s_r})^{2+n/r'-n/r}$,
  where $M=M(r,s,n)\geq1$ is fixed in the proof. Then if
  $\rho\in\C^n$, $\abs{\Im\rho} \geq I_0$ there is $\psi\in
  H^s_{r'}(\R^n)$ satisfying
  \[
  (\Delta+2\rho\cdot\nabla + q)\psi = f,
  \]
  \[
  \norm{\psi}_{H^s_{r'}(\R^n)} \leq 2 M \abs{\Im\rho}^{-(2 +
    \frac{n}{r'} - \frac{n}{r})} \norm{f}_{H^s_r(\R^n)}.
  \]

  There is also $p\geq2$ and a Sobolev embedding constant $E =
  E(s,n,r)\geq1$ such that
  \[
  \norm{\psi}_{L^p(\R^n)} \leq E M \abs{\Im\rho}^{-(2 + \frac{n}{r'} -
    \frac{n}{r})} \norm{f}_{H^s_r(\R^n)}.
  \]
  We have the following observations about the choice of $p$ and the
  decay rate of $\psi$ compared to $\abs{\Im\rho}^{-n/p}$.
  \begin{enumerate}
    \item If $s>\frac{n}{r'}$ then $p=\infty$ and
      $2+\frac{n}{r'}-\frac{n}{r} > \frac{n}{p}$,
    \item if $s=\frac{n}{r'}$ then we may choose any finite $p$ such
      that $\frac{1}{p} < \frac{2}{n} + \frac{1}{r'} - \frac{1}{r}$
      which is positive, and then $2+\frac{n}{r'}-\frac{n}{r} >
      \frac{n}{p}$,
    \item if $\frac{n}{r}-2 < s < \frac{n}{r'}$ then $s-\frac{n}{r'} =
      -\frac{n}{p}$ and $2+\frac{n}{r'}-\frac{n}{r} > \frac{n}{p}$,
      and finally
    \item if $s\leq \frac{n}{r}-2$ then $s-\frac{n}{r'} = -
      \frac{n}{p}$ but $2+\frac{n}{r'}-\frac{n}{r} \leq \frac{n}{p}$.
  \end{enumerate}
  
  Lastly, if $f\in L^2_{loc}$ and $q\in L^\infty_{loc}$ then given any
  bounded domain, for example $B_{3R}$, we have the elliptic
  regularity estimate
  \[
  \norm{\psi}_{H^2(B_{2R})} \leq C_R \big(\norm{f}_{L^2(B_{3R})} + (1
  + \abs{\rho}^2 + \norm{q}_{L^\infty(B_{3R})})
  \norm{\psi}_{L^p(\R^n)}\big)
  \]
  where $C_R$ depends only on $R$.
\end{lemma}

\begin{proof}
  Fix $M<\infty$ as the $\rho$-independent constant in the estimate
  \[
  \norm{f}_{L^{r'}(\R^n)} \leq M \abs{\Im\rho}^{n(1/r-1/r')-2}
  \norm{(\Delta+2\rho\cdot\nabla)f}_{L^r(\R^n)}
  \]
  by \cite{KRS} or in Theorem 5.4 in the notes \cite{Ruiz}. By
  Proposition 3.3 in \cite{PSV} the equation
  \[
  (\Delta +2\rho\cdot\nabla + q)\psi = f
  \]
  has a solution $\psi\in H^s_{r'}(\R^n)$ when $\abs{\Im \rho} \geq
  I_0$. Moreover it satisfies
  \[
  \norm{\psi}_{H^s_{r'}(\R^n)} \leq 2M \abs{\Im\rho}^{-(2+n/r'-n/r)}
  \norm{f}_{H^s_r(\R^n)}.
  \]
  Sobolev embedding implies the $L^p$ estimates in the four cases of
  the statement. Note that in each case we have $p\geq r'>2$.
  
  The elliptic regularity estimate needs some work. First assume that
  $G\in H^s(\R^n)$, $F\in H^s(\R^n)$ and
  $(\Delta+2\rho\cdot\nabla)G=F$.  Then
  \begin{align*}
    &\norm{G}_{H^{s+2}(\R^n)} = \norm{(1+\abs{\xi}^2)^{s/2}
      (1+\abs{\xi}^2) \hat G}_{L^2(\R^n)} \\ &\qquad =
    \norm{(1+\abs{\xi}^2)^{s/2}( \hat G + 2i\rho\cdot\xi \hat G - \hat
      F )}_{L^2(\R^n)} \\ &\qquad \leq \norm{G}_{H^s(\R^n)} +
    \norm{F}_{H^s(\R^n)} + 2\norm{(1+\abs{\xi}^2)^{s/2} \rho\cdot\xi
      \hat G}_{L^2(\R^n)}
  \end{align*}
  because $(-\abs{\xi}^2 + 2i\rho\cdot\xi)\hat G = \hat F$. By looking
  at what happens when $\abs{\xi}$ is larger or smaller than
  $3\abs{\rho}$ we see that $\abs{\rho\cdot\xi} \leq
  \lvert-\abs{\xi}^2 + 2i\rho\cdot\xi\rvert + 3 \abs{\rho}^2$. Hence
  \begin{equation}\label{constantCoefApriori}
    \norm{G}_{H^{s+2}(\R^n)} \leq 3\norm{F}_{H^s(\R^n)} +
    (1+6\abs{\rho}^2)\norm{G}_{H^s(\R^n)}.
  \end{equation}
  
  Now let $\chi,\tilde\chi \in C^\infty_0(\R^n)$ such that $\tilde\chi
  \equiv 1$ on $\supp \chi$. Assume that $f\in L^2_{loc}$ and $q\in
  L^\infty_{loc}$. Next
  \begin{equation}\label{localizedFaddeev}
    (\Delta+2\rho\cdot\nabla)(\chi\psi) = \chi (f-q\psi) +
    2\nabla\chi\cdot\nabla(\tilde\chi\psi) + (\Delta\chi +
    2\rho\cdot\nabla\chi)\psi
  \end{equation}
  in the distribution sense. We have $q\psi \in L^p_{loc}$, $p\geq2$
  so $\chi q\psi \in L^2(\R^n)$. Similarly $\tilde\chi\psi\in
  L^2(\R^n)$ and so $\nabla\chi\cdot\nabla(\tilde\chi\psi)\in
  H^{-1}(\R^n)$. The last term on the right-hand side is in
  $L^2(\R^n)$. By absorbing all the norms of $\chi$, $\tilde\chi$ into
  a constant we get the estimate
  \[
  C_{\chi,\tilde\chi,p} \big(\norm{f}_{L^2(\supp \chi)} +
  (1+\abs{\rho} + \norm{q}_{L^\infty(\supp\chi)})
  \norm{\psi}_{L^p(\R^n)}\big)
  \]
  for the $H^{-1}(\R^n)$-norm of the right-hand side. By
  \eqref{constantCoefApriori} and since $\psi\in L^p$, $p\geq2$,
  \[
  \norm{\chi\psi}_{H^1(\R^n)} \leq \tilde C_{\chi,\tilde\chi,p}
  \big(\norm{f}_{L^2(\supp\chi)} + (1+\abs{\rho} + \abs{\rho}^2 +
  \norm{q}_{L^\infty(\supp\chi)}) \norm{\psi}_{L^p(\R^n)}\big)
  \]
  and this is true no matter the choice of $\chi,\tilde\chi\in
  C^\infty_0(\R^n)$, $\tilde\chi\equiv 1$ on $\supp\chi$.
  
  Consider the bounded domain $B_{2R}$ now. Take a chain of cut-off
  functions $\chi,\tilde\chi,\overline\chi \in C^\infty_0(B_{3R})$
  such that $\overline\chi\equiv1$ on $\supp \tilde\chi$,
  $\tilde\chi\equiv1$ on $\supp\chi$ and finally $\chi\equiv1$ on
  $B_{2R}$. Then $\chi\psi\in H^2(\R^n)$ according to
  \eqref{constantCoefApriori} if the right-hand side of
  \eqref{localizedFaddeev} is in $L^2(\R^n)$. But this is indeed true
  by going through the previous paragraph while substituting
  $(\tilde\chi,\overline\chi)$ for $(\chi,\tilde\chi)$. This gives the
  final estimate
  \begin{align*}
    &\norm{\psi}_{H^2(B_{2R})} \leq \norm{\chi\psi}_{H^2(\R^n)}
    \\ &\qquad \leq \mathcal{C}_{\chi,\tilde\chi,\overline\chi,p}
    \big(\norm{f}_{L^2(\supp\tilde\chi)} + (1+\abs{\rho} +
    \abs{\rho}^2 + \norm{q}_{L^\infty(\supp\tilde\chi)})
    \norm{\psi}_{L^p(\R^n)}\big)
  \end{align*}
  which can be bounded above by the estimate of the statement. Note
  that the test functions can be chosen based exclusively on the set
  $B_{2R}$, and their norms have a finite supremum while $p$ explores
  the whole set ${[{2,\infty}]}$. Hence the constant can be made to
  depend only on $R$.
\end{proof}

The next estimate concerns a potential consisting of a
H\"older-continuous function multiplying the characteristic function
of a polytope. For a clearer notation we define a multiplier norm for
a polytope first.

\begin{definition}
  A set $P \subset \R^n$ is a \emph{bounded open polytope} if $P$ is
  bounded, open and $\overline P$ is a finite union of finite
  intersections of closed half-spaces.
\end{definition}

\begin{definition}\label{triangDef}
  Let $P \subset \R^n$ be a bounded open polytope. We say a collection
  $\{ H_{jl} \mid j=1,\ldots,J,\, l=1,\ldots,L_j \}$ of half-spaces is
  a \emph{triangulation} of $P$ if $J \in \N$, $L_1,\ldots,L_J \in
  \N$, $H \subset H_{jl} \subset \overline H$ for some open half-space
  $H\in\R^n$, the intersections $\bigcap_l H_{jl}$ are disjoint for
  different $j$, and
  \[
  P = \bigcup_{j=1}^J \bigcap_{l=1}^{L_j} H_{jl}.
  \]
  If $s\in\R$ and $1\leq r <\infty$ let $C_{s,r} \in \R \cup
  \{+\infty\}$ be the norm of the map $H^s_r(\R^n) \to H^s_r(\R^n)$,
  $f \mapsto \chi_H f$, where $H\subset\R^n$ is a half-space. Then by
  $\norm{P}_{T(s,r)}$ we mean
  \begin{equation}
    \norm{P}_{T(s,r)} = \inf \left\{ \sum_{j=1}^J C_{s,r}^{L_j}
    \,\middle|\, \text{$(H_{jl})_{j,l}$ is a triangulation of
      $P$}\right\}.
  \end{equation}
\end{definition}

\begin{lemma}\label{Pmult}
  Let $P\subset\R^n$ be a bounded open polytope, $s \geq 0$, $r\geq1$
  and $sr<1$. Then $\norm{P}_{T(s,r)} < \infty$ and $\norm{\chi_P
    f}_{H^s_r(\R^n)} \leq \norm{P}_{T(s,r)} \norm{f}_{H^s_r(\R^n)}$.
  Moreover we have $\norm{P}_{T(s_0,r)} \leq \norm{P}_{T(s_1,r)}$ if
  $s_0 \leq s_1$.
\end{lemma}

\begin{proof}
  By definition $P$ has a finite triangulation of let us say
  $m<\infty$ simplices. Each simplex in $\R^n$ is the intersection of
  $n+1$ half-spaces. By Triebel \cite{Triebel1}, Section 2.8.7, the
  map $f \mapsto \chi_H f$ is bounded in $H^s_r(\R^n)$ under the
  conditions for $s$ and $r$ given. Hence $\norm{P}_{T(s,r)} \leq m
  C_{s,r}^{n+1} < \infty$. If $(H_{jl})_{j,l}$ is a triangulation,
  then the intersections $\bigcap_{l=1}^{L_j} H_{jl}$ are disjoint, so
  $\chi_P = \sum_{j=1}^J \prod_{l=1}^{L_j} \chi_{H_{jl}}$ and thus
  $\norm{\chi_P f}_{H^s_r} \leq \sum_{j=1}^J C_{s,r}^{L_j}
  \norm{f}_{H^s_r}$. The multiplier estimate follows by taking the
  infimum over all triangulations. The last claim follows since
  complex interpolation of Sobolev spaces implies that $C_{s_0,r} \leq
  C_{s_1,r}$ if $s_0\leq s_1$.
\end{proof}

\begin{lemma}\label{potNorm}
  Let $V = \chi_P \varphi$ with $P \subset B_R$ an open polytope and
  $\varphi \in C^\alpha(\R^n)$ with $\alpha > 0$. Let $0\leq s <
  \alpha$, $1 \leq r \leq 2$ and $sr<1$. Then $V \in H^s_r(\R^n)$,
  \[
  \norm{V}_{H^s_r(\R^n)} \leq C_{\alpha,s,r,R} \norm{P}_{T(s,r)} 
  \norm{\varphi}_{C^\alpha(\R^n)}
  \]
  and
  \[
  \norm{Vf}_{H^s_r(\R^n)} \leq C_{\alpha,s,r,R} \norm{P}_{T(s,r)}
  \norm{\varphi}_{C^\alpha(\R^n)} \norm{f}_{H^s_{r'}(\R^n)}
  \]
  where $1/r+1/r'=1$ and $\norm{P}_{T(s,r)}$ is defined in
  Definition~\ref{triangDef}.
\end{lemma}

\begin{proof}
  Let $\Phi \in C^\infty_0$ be such that $\Phi = 1$ on $B_R$. Then we
  have the representation
  \[
  V = \chi_P \varphi \Phi
  \]
  which helps us prove the estimates.

  By the last corollary of Section 4.2.2 in \cite{Triebel2} there is a
  finite upper bound $C_{\alpha,s,r}$ for the pointwise multiplier
  operator norm of any $C^\alpha$ function multiplying in
  $H^s_r(\R^n)$ when $s<\alpha$. Then the first claim
  \[
  \norm{V}_{H^s_r(\R^n)} \leq C_{\alpha,s,r} 
  \norm{\varphi}_{C^\alpha(\R^n)} \norm{\chi_P \Phi}_{H^s_r(\R^n)} \leq 
  C_{\alpha,s,r,\Phi} \norm{P}_{T(s,r)} \norm{\varphi}_{C^\alpha(\R^n)}
  \]
  follows from Lemma~\ref{Pmult} since $\norm{P}_{T(s,r)} < \infty$ by
  $s\geq0$, $r\geq1$ and $sr<1$.
  
  By \cite{PSV} Proposition 3.5 or \cite{Behzadan--Holst} Theorem 7.5
  the product of a $H^s_{r'}(B_R)$ and $H^s_{r/(2-r)}(B_R)$ function
  is in $H^s_r(B_R)$ when $s\geq0$ and $1\leq r\leq2$. According to
  \cite{Triebel2}, we know that $C^\alpha$-functions are pointwise
  multipliers for $H^s_{r/(2-r)}$ too. The last claim
  \begin{align*}
    &\norm{Vf}_{H^s_r(\R^n)} \leq \norm{P}_{T(s,r)} \norm{\varphi \Phi
      f}_{H^s_r(B_R)} \\ &\qquad\leq M_{s,r} \norm{P}_{T(s,r)}
    \norm{\varphi \Phi}_{H^s_\frac{r}{2-r}(B_R)}
    \norm{f}_{H^s_{r'}(B_R)} \\ &\qquad \leq C_{\alpha,s,r,\Phi}
    \norm{P}_{T(s,r)} \norm{\varphi}_{C^\alpha(\R^n)}
    \norm{f}_{H^s_{r'}(\R^n)}
  \end{align*}
  follows then because $V$ is supported in $\overline B_R$.
\end{proof}

We are now ready to specialise previous lemmas into proving the
existence of the complex geometrical optics solutions in the context
of corner scattering in two and three dimensions.
  
The conditions on the H\"older smoothness index $\alpha$ of the
following proposition follow from various requirements: For the
half-space multipliers we needed $sr<1$ and $s<\alpha$. To have good
enough error decay estimates for $\psi$ from Lemma~\ref{FaddeevSol} we
need $s > n/r-2$. Combining these gives $n-2r<sr<1$ i.e. $r >
(n-1)/2$. On the other hand we must have $1/r - 1/r' \geq 2/(n+1)$
i.e. $r \leq 2(n+1)/(n+3)$ in Lemma~\ref{FaddeevSol}. These two
inequalities have solutions only when $n\in\{2,3\}$. The use of these
solutions for corner scattering in higher dimensions requires the
Fourier transforms of Besov spaces \cite{BPS}.

Since $\alpha$ is the parameter that ultimately decides which
potentials are admissible, we want a largest possible range for
it. This is achieved by making $s$, and thus $n/r-2$, as small as
possible. Hence $r$ must be largest, and a fortiori we choose
$r=2(n+1)/(n+3)$.

\begin{proposition}\label{CGOsol}
  Let $n\in\{2,3\}$ and $0 \leq s < 5/6$ in 2D or $1/4 < s < 3/4$ in
  3D.  Let $\varphi \in C^\alpha(\R^n)$ with $\alpha > s$ and
  $\norm{\varphi}_{C^\alpha} \leq \mathcal M$. Let $P \subset B_R$ be
  a bounded open polytope, $r=2(n+1)/(n+3)$, and assume that
  $\norm{P}_{T(s,r)} \leq \mathcal D$.
    
  Let $k>0$ and set $V = \chi_P \varphi$. Then there is $p\geq2$ and
  $C_{\alpha,s,n,R}<\infty$ with the following properties. If
  $\rho\in\C^n$, $\rho\cdot\rho+k^2 = 0$, $\abs{\Im\rho} \geq
  (C_{\alpha,s,n,R} k^2 \mathcal D \mathcal M)^{(n+1)/2}$, then there
  is $\psi \in L^p(\R^n)$ such that $u_0(x) = \exp(\rho\cdot
  x)(1+\psi(x))$ satisfies
  \[
  (\Delta+k^2(1+V))u_0 = 0
  \]
  in $\R^n$, and
  \[
  \norm{\psi}_{L^p(\R^n)} \leq C_{\alpha,s,n,R} k^2 \mathcal D
  \mathcal M \abs{\Im\rho}^{-n/p-\beta}
  \]
  with $\beta=\beta(s,n) > 0$. Moreover $\psi\in H^2(B_{2R})$ with
  norm estimate
  \[
  \norm{\psi}_{H^2(2R)} \leq C_{\alpha,s,n,R}(1+\abs{\rho}^2 + (1+k^2)
  \mathcal M).
  \]
\end{proposition}

\begin{proof}
  Set $q = k^2 V$ and $f = -k^2 V$. Now $0\leq s<\alpha$, $1 \leq r
  \leq 2$ and $sr<1$, so by Lemma~\ref{potNorm} we have
  \[
  \norm{f}_{H^s_r(\R^n)}, \norm{m_q}_{H^s_{r'} \to H^s_r} \leq
  C_{\alpha,s,n,R} k^2 \norm{P}_{T(s,r)}
  \norm{\varphi}_{C^\alpha(\R^n)}
  \]
  where $m_q$ is the pointwise multiplier operator.
  
  We have $1/r-1/r'=2/(n+1)$, $r\leq2$. The lower bound for
  $\abs{\Im\rho}$ matches Lemma~\ref{FaddeevSol} so we have existence
  of $\psi$. The condition $s > n/r-2$ that's required for the good
  enough error term decay is also satisfied by our a-priori
  requirements on $s$.
  
  For the $H^2$-norm estimate note that $I_0 = (C_{\alpha,s,n,R} k^2
  \mathcal D \mathcal M)^{(n+1)/2}$ and the bound for
  $\norm{f}_{H^s_r}$ imply that $\norm{\psi}_p \leq C_{s,n}$. We also
  see that $\norm{f}_{L^2} \leq C_{n,R} \mathcal M$ by its definition.
\end{proof}

\section{Stability proofs} \label{sect:proofs}

The proofs of the following two lemmas are in the appendix.

\begin{lemma} \label{Qangle2D}
  Let $P, P' \subset \R^2$ be two open bounded convex polygons. Let
  $Q$ be the convex hull of $P \cup P'$. If $x_c$ is a vertex of $P$
  such that $d(x_c, P') = d_H(P,P')$, where $d_H$ gives the Hausdorff
  distance,
  \[
  d_H(P,P') = \max\big( \sup_{x\in P} d(x,P'), \, \sup_{x'\in P'}
  d(P,x')\big),
  \]
  then $x_c$ is a vertex of $Q$. If the angle of $P$ at $x_c$ is
  $\alpha$, then the angle of $Q$ at $x_c$ is at most $(\alpha+\pi)/2
  < \pi$.
\end{lemma}

\begin{lemma} \label{Qangle3D}
  Let $P, P' \subset \R^3$ be two open cuboids. Let $Q$ be the convex
  hull of $P \cup P'$. If $x_c$ is a vertex of $P$ such that $d(x_c,
  P') = d_H(P,P')$, where $d_H$ gives the Hausdorff distance,
  \[
  d_H(P,P') = \max\big( \sup_{x\in P} d(x,P'), \, \sup_{x'\in P'}
  d(P,x')\big),
  \]
  then $x_c$ is a vertex of $Q$. The latter can also fit inside an
  open spherical cone $\mathcal Q$ with vertex $x_c$ and opening angle
  $2\alpha' < \pi$. Here $\alpha'$ is independent of $P$ and $P'$ or
  their location.
\end{lemma}

We are ready to proof the final theorem whose statement is on page
\pageref{potSupportStab}.

\begin{proof}[Proof of Theorem~\ref{potSupportStab}]
  By Lemma~\ref{Qangle2D} and Lemma~\ref{Qangle3D} and possibly
  switching the symbols $P$ and $P'$ (and their associated waves and
  potentials) we may assume that $\mathfrak h = d(x_c,P')$ with $x_c$
  a vertex of $P$.  We use the total wave $u'$ of the second potential
  $V'$ as a ``local incident wave'' in the neighbourhood of
  $x_c$. This is allowed since $(\Delta + k^2) u' = 0$ there because
  $V'=0$ around $x_c$.

  The potentials $V$ and $V'$ give well-posed scattering. Denote the
  $L^2$-norm of the difference of the far-field patterns by
  $\varepsilon = \norm{u^s_\infty-u'^s_\infty}_{L^2(\mathbb
    S^{n-1})}$. Use Proposition~\ref{FF2bndry}. If $Q$ is the convex
  hull of $P \cup P'$ then
  \begin{equation} \label{deltaDef}
    \sup_{\partial Q} \big( \abs{u^s-u'^s} + \abs{\nabla (u^s-u'^s)}
    \big) \leq \frac{C}{\sqrt{\ln\ln \frac{\mathcal S}{\varepsilon}}}
  \end{equation}
  when $\varepsilon < \varepsilon_m$. Here $C$ and $\varepsilon_m$
  depend only on the a-priori parameters. Denote the right-hand side
  by $\delta(\varepsilon)$ to conserve space in formulas.
  
  Let $\mathfrak Q$ be the polyhedral cone generated by the convex
  hull $Q$ at $x_c$. By Lemma~\ref{Qangle2D} and Lemma~\ref{Qangle3D}
  there is an open spherical cone $\mathcal Q \supset \mathfrak Q
  \supset Q$ with vertex $x_c$ having opening angle at most $2\alpha'
  = 2\alpha'(\alpha_m, \alpha_M) < \pi$. Let $\mathfrak P$ be the cone
  generated by $P$ at its vertex $x_c$. Remember for later that
  $(\mathcal Q, \mathfrak P) \in \mathcal G(\alpha_m, \alpha_M,
  \alpha', n)$ using the notation from Lemma~\ref{lowerBound}.

  Let $h = \min(\ell, \mathfrak h)$ and it is enough to consider the
  case $h>0$. We have $P \cap B(x_c,h) = \mathfrak P \cap B(x_c,h)$
  and $Q \cap B(x_c,h) = \mathfrak Q \cap B(x_c,h)$. Denote the former
  by $P_h$ and the latter by $Q_h$. We also have $P_h \cap P' = Q_h
  \cap P' = \emptyset$.

  We want to use Proposition~\ref{intByPartsEstim} next. The
  conditions of non-vanishing total waves of
  Definition~\ref{nonVanishing} imply that we have $N=0$, $P_N(x)
  \equiv u'(x_c) \neq 0$. Moreover, as in the proof of
  Proposition~\ref{FF2bndry}, we see that $u'$ is Lipschitz with norm
  at most $C(k,R,\mathcal M,\mathcal S)$. The other conditions of
  Proposition~\ref{intByPartsEstim} are also satisfied. Recall also
  $\delta(\varepsilon) = C/\sqrt{\ln\ln(\mathcal S/\varepsilon)}$ from
  \eqref{deltaDef}, and that $\norm{u},\norm{u'} \leq C_{k,R,\mathcal
    S}$ in $H^2(B_{2R})$. We can absorb this constant into the
  constants of the inequality. Hence there is a constant $C$ depending
  only on a-priori parameters such that if $1/p+1/p'=1$, then
  \begin{align}
  &C\abs{\varphi(x_c) \int_{\mathfrak P} e^{\rho\cdot(x-x_c)} u'(x_c)
      dx} \leq \abs{\Re\rho}^{-n} e^{-\delta_0\abs{\Re\rho}h/2}
    \notag\\ &\qquad\quad + \abs{\Re\rho}^{-n-\min(1,\alpha)} +
    \abs{\Re\rho}^{-n/p'} \norm{\psi}_{L^p(P_h)} \notag\\ &\qquad
    \quad + h^{(n-1)/2} \abs{\rho} (1+\norm{\psi}_{H^2(B_{2R})})
    \delta(\varepsilon) \notag\\ &\qquad \quad + h^{n/2-1}
    e^{-\delta_0\abs{\Re\rho}h} \abs{\rho}
    (1+\norm{\psi}_{H^2(B_{2R})}) \label{bigEstim}
  \end{align}
  whenever $u_0\in H^2(B_{2R})$ satisfies $(\Delta+k^2(1+V))u_0=0$,
  \[
  u_0(x) = e^{\rho\cdot(x-x_c)}(1+\psi(x)),
  \]
  $\psi \in L^p$ in $Q_h$ with $\rho \in \C^n$, $\abs{\rho}\geq1$ and
  $\Re \rho \cdot (x-x_c) \leq -\delta_0 \abs{\Re\rho} \abs{x-x_c}$
  for some $\delta_0>0$ and any $x\in Q_h$.

  Recall that $(\mathcal Q, \mathfrak P) \in \mathcal G(\alpha_m,
  \alpha_M, \alpha', n)$, and hence we may use Lemma~\ref{lowerBound}.
  It gives us constants $\tau_0 = \tau_0(k,\alpha_m, \alpha_M,
  \alpha', n)$, $c = c(\alpha_m, \alpha_M, n) >0$ and a curve $\tau
  \mapsto \rho(\tau) \in \C^n$, $\tau = \abs{\Re\rho(\tau)}$
  satisfying the conditions required of $\rho$ above with $\delta_0 =
  \cos \alpha' > 0$, $\rho(\tau) \cdot \rho(\tau) + k^2 = 0$ and
  \begin{equation} \label{lbEstim}
    \abs{\int_{\mathfrak P} e^{\rho\cdot(x-x_c)} u'(x_c) dx} \geq c
    \abs{u'(x_c)} \tau^{-n}
  \end{equation}
  whenever $\tau \geq \tau_0$.

  If $\tau \geq \max(\tau_0, C_0)$, with the constant $C_0$ depending
  on a-priori parameters and arising from Proposition~\ref{CGOsol},
  then the latter gives existence of $u_0$ and $\psi$ required
  above. We may indeed use that proposition because the a-priori
  bounds on the H\"older smoothness index $\alpha$ imply the existence
  of a suitable Sobolev smoothness index $s$ used in there. Finally it
  gives the estimates
  \[
  \norm{\psi}_{L^p(\R^n)} \leq C \abs{\Im\rho}^{-n/p-\beta}
  \]
  for some $\beta = \beta(s,n) > 0$ and
  \[
  \norm{\psi}_{H^2(B_{2R})} \leq C(1+\abs{\rho}^2)
  \]
  where $C$ again depends only on the a-priori parameters.

  We have all the fundamental estimates now. Let us apply them. We
  have $\exp(-x) \leq x^{-1}$ and $\exp(-x) \leq (n+4)!\, x^{-n-4}$
  for all $x>0$. Also, since $\rho(\tau) \cdot \rho(\tau) + k^2 = 0$,
  we get $\abs{\rho(\tau)} = \sqrt{k^2 + 2\tau^2}$. By taking a new
  lower bound for $\tau$, for example $\tau \geq k$, we may assume
  that $\abs{\rho(\tau)} \leq \sqrt{3} \tau$. Hence we can estimate
  \begin{align*}
    &\abs{\Re\rho}^{-n} e^{-\delta_0 \abs{\Re\rho} h/2} \leq
    C\abs{\Re\rho}^{-n-1} h^{-1},\\ &\abs{\Re\rho}^{-n/p'}
    \norm{\psi}_{L^p} \leq C \abs{\Re\rho}^{-n-\beta},\\ &h^{(n-1)/2}
    \abs{\rho}(1+\norm{\psi}_{H^2(B_{2R})}) \delta(\varepsilon) \leq C
    h^{(n-1)/2} \abs{\Re\rho}^3 \delta(\varepsilon),\\ & h^{n/2-1}
    e^{-\delta_0\abs{\Re\rho}h} \abs{\rho}
    (1+\norm{\psi}_{H^2(B_{2R})}) \leq C h^{-n/2-5}
    \abs{\Re\rho}^{-n-1}
  \end{align*}
  in \eqref{bigEstim}. Divide the new constants to the left hand side,
  take the lower bound \eqref{lbEstim} into account and use the
  a-priori assumption $\abs{u'(x)} \geq c > 0$ in $B_R \setminus
  P'$. Finally, using $h\leq1$ and $\tau\geq1$ we get
  \[
  c \abs{\varphi(x_c)} \leq h^{(n-1)/2} \big( \delta(\varepsilon)
  \tau^{n+3} + h^{-n-9/2} \tau^{-m} \big)
  \]
  where $m = \min(1,\alpha,\beta)$. This holds as long as $\tau \geq
  \max(\tau_0, C_0, k)$ and $h = \min(\ell,\mathfrak h)$. To make
  formulas simpler we estimate the right-hand side above and get
  \begin{equation} \label{allInSameInequality}
    c \abs{\varphi(x_c)} \leq \delta(\varepsilon) \tau^{n+5} +
    h^{-n-5} \tau^{-m}.
  \end{equation}
  
  Setting $\tau = \tau_e$ with
  \[
  \tau_e = \left( \frac{1}{h^{n+5}\delta(\varepsilon)}
  \right)^{\frac{1}{m+n+5}}
  \]
  makes both terms on the right hand side of
  \eqref{allInSameInequality} equal (which gives the minimum modulo
  constants), and the inequality becomes
  \begin{equation} \label{allInSameOptimized}
    c \abs{\varphi(x_c)} \leq 2 h^{-\frac{(n+5)^2}{m+n+5}}
    \delta(\varepsilon)^{\frac{m}{m+n+5}}.
  \end{equation}
  Note that if $\varepsilon$ is small enough, then
  \[
  \tau_e \geq \tau_e h^{\frac{n+5}{m+n+5}} =
  (\delta(\varepsilon))^{-\frac{1}{m+n+5}} \geq \max(\tau_0,C_0,k)
  \]
  and so we can choose $\tau=\tau_e$ in \eqref{allInSameOptimized}.
  Solving for $h$ in it gives
  \[
  \min(\ell,\mathfrak h) = h \leq C \delta^{\frac{m}{(n+5)^2}}
  \abs{\varphi(x_c)}^{-\frac{m}{(n+5)^2} - \frac{1}{n+5}}.
  \]
  By the a-priori bounds of Definition~\ref{aprioriBounds} we have
  $\abs{\varphi(x_c)} \geq \mu > 0$. Hence if $\varepsilon$ is again
  small enough (now also depending on $\mu$ and $\ell$), then the
  right-hand side is smaller than $\ell$, and so $\min(\ell,\mathfrak
  h) = \mathfrak h$. Writing out the definition of
  $\delta(\varepsilon)$ gives
  \[
  \mathfrak h \leq C \left(\ln\ln
  \frac{S}{\varepsilon}\right)^{-\frac{m}{2(n+5)^2}}
  \]
  and the claim is proven.
\end{proof}

\begin{proof}[Proof of Theorem~\ref{cornerScatStab}]
  The proof uses the same lemmas and propositions as the proof of
  Theorem~\ref{potSupportStab}. Now instead of having two non-trivial
  potentials $V$ and $V'$, we choose the following: $P' = \emptyset$,
  $V' \equiv 0$. This implies that $u' = u^i$, $u'^s \equiv 0$,
  $u'^s_\infty \equiv 0$ among others. In particular $V' \equiv 0$ is
  trivially admissible.
  
  Proceed as in the proof of Theorem~\ref{potSupportStab}, except that
  choose $h = \ell$ instead of $h = \min(\ell, d_H(P,P'))$. Up to
  showing \eqref{allInSameOptimized} none of the constants depend on
  $\mu$ or $\ell$. Now, if $\varepsilon$ is small enough, let's say at
  most $\varepsilon_{min}$ which depends only on a-priori parameters
  except for $\ell$, $\varphi(x_c)$, then
  \[
  (\delta(\varepsilon))^{-\frac{1}{m+n+5}} \geq \max(\tau_0,C_0,k)
  \]
  and we can again let $\tau=\tau_e$ in \eqref{allInSameOptimized}.
  Solving for $\varepsilon$ in it gives
  \[
  \norm{u^s_\infty}_{L^2(\mathbb S^{n-1})} = \varepsilon \geq
  \frac{\mathcal S}{\exp \exp (C \ell^{-2/\gamma}
    \abs{\varphi(x_c)}^{-2-2/((n+5)\gamma)})}
  \]
  for $\gamma = \min(1,\alpha,\beta)/(n+5)^2$ as in the previous
  proof, and a constant $C$ depending on a-priori data but not $\ell$
  or $\varphi(x_c)$. If on the other hand $\varepsilon >
  \varepsilon_{min}$ the claim is immediately true.
\end{proof}

\section{Appendix} \label{sect:appendix}

\begin{proof}[Proof of Lemma~\ref{Qangle2D}]
  Let $a$ and $b$ be the vertices of $P$ on the adjacent edges to
  $x_c$.  Let $C \in \overline{P'}$ be any point such that $d(x_c,C) =
  d_H(P,P')$, and let $h = d_H(P,P')$. Consider the circle $S(x_c,h)$.
  Let $H_a$ be an open half-plane tangent to $S(x_c,h)$, parallel to
  the segment $x_c a$ and such that it is on the opposite side of $x_c
  a$ than $b$. Construct $H_b$ similarly. See Figure~\ref{pic1}a. Let
  $H_C$ be the closed half-space tangent to $S(x_c,h)$ at $C$ with
  $x_c \notin H_C$.
  
  Let $x' \in P'$. If $x' \in H_a$, then $d(x',P) \geq
  d(x',\ell_{x_c,a}) > h$ where $\ell_{x_c,a}$ is a line through $x_c$
  and $a$. This follows from the convexity of $P$: the polygon is
  contained in the cone with vertex $x_c$ and edges defined by $a$ and
  $b$. Thus $d_H(P,P') \geq d(x',P) > h = d_H(P,P')$, a contradiction.
  Similarly for $x' \in H_b$. Consider $H_C$ next: the convexity of
  $P'$ implies that the segment $x'C$ belongs to $\overline{P'}$. If
  $x'\notin H_C$, then there is $y' \in x'C \cap B(x_c,h)$ by the
  non-tangency of $x'C$. Then $y' \in \overline{P'}$ and $d(x_c,y') <
  h$ so $d_H(P,P') < h$, a contradiction again. Thus we see that $P'
  \subset H_a^\complement \cap H_b^\complement \cap H_C$.
  \begin{figure}
    \begin{center}
      \includegraphics[]{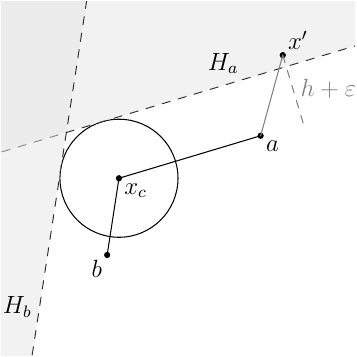} \includegraphics[]{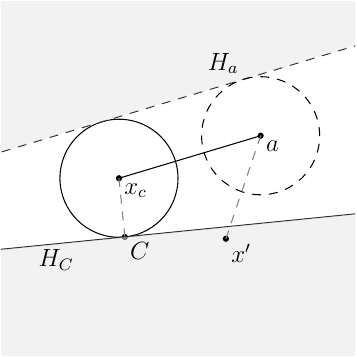}
    \end{center}
    \caption{a) $P' \subset H_a^\complement \cap H_b^\complement \cap
      H_C$, \quad b) ray $x_c$ to $a$ must meet $H_C$}
    \label{pic1}
  \end{figure}

  Next, $H_C$ must be distance $h$ from $a$: if it were not, then for
  any $x' \in P'$ we have $d(a,x') \geq d(a,H_C) > h$ since $P'
  \subset H_C$ as was shown above. Hence $\partial H_C$ and $\partial
  H_a$ are either parallel (a case we skip in this proof) or meet at a
  point $A'$, in which case the ray from $x_c$ towards $a$ intersects
  $H_C$.  Do the same for $b$ to get $B'$. See
  Figure~\ref{pic1}b. This means that $S(x_c,h)$ is the incircle of
  the triangle formed by $H_a$, $H_b$ and $H_C$.

  We can now see that $x_c$ is a vertex of $Q$. First of all $x_c \in
  \overline Q$ since $x_c \in \overline P$. Also, $P$ is inside the
  angle $a x_c b$ and $P'$ inside the angle $A' x_c B'$, which is
  obviously less than $\pi$. Thus $x_c$ is a vertex of $Q$. Moreover
  its angle is at most $\angle A' x_c B'$. See Figure~\ref{pic3}a.
  \begin{figure}
    \begin{center}
      \includegraphics[]{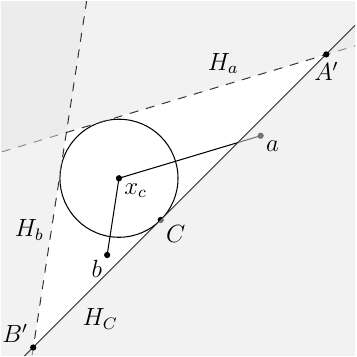} \includegraphics[]{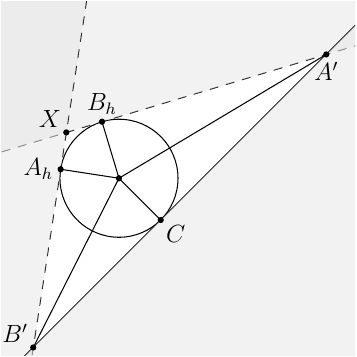}
    \end{center}
    \caption{a) $S(x_c,h)$ is an incircle, \quad b) solving $\angle A'
      x_c B'$}
    \label{pic3}
\end{figure}

  Let $X$ be the intersection of $\partial H_a$ and $\partial
  H_b$. This is a well-defined point since $0 < \angle a x_c b <
  \pi$. We have $\angle A' X B' = \angle a x_c b = \alpha$ by parallel
  transport of $x_c a$ to $X A'$ and $x_c b$ to $X B'$. Let the
  perpendiculars from $x_c$ to $X A'$, $A' B'$, $B' X$ have base
  points $B_h$, $C$, $A_h$, respectively. See Figure~\ref{pic3}b. Then
  $\angle A_h x_c B_h = \pi - \alpha$, $\angle B_h x_c A' = \angle A'
  x_c C$ and $\angle C x_c B' = \angle B' x_c A_h$. This implies that
  $\angle A' x_c B' = (\alpha+\pi)/2$ at once since the sum of all of
  these angles is $2\pi$.
\end{proof}

\begin{proof}[Proof of Lemma~\ref{Qangle3D}]
  The proof proceeds as in the proof of Lemma~\ref{Qangle2D}. We can
  choose coordinates such that $x_c = 0$ and the three edges of $P$
  starting from $x_c$ lie on the positive coordinate axes having unit
  vectors$e_1$, $e_1$ and $e_3$. Let $h = d(x_c,C) = d_H(P,P')$ for
  some $C \in \overline P'$.

  If we set $H_j = \{ x \mid x\cdot e_j < -h\}$, then as in the 2D
  proof, we see that $P' \subset H_j^\complement$. Similarly, if $H_C$
  is the closed half-space tangent to $S(x_c,h)$ at $C$, we see that
  $P' \subset H_C$. Hence $P' \subset H_1^\complement \cap
  H_2^\complement \cap H_3^\complement \cap H_C$.

  If $C_3 < 0$, i.e. it is on the lower hemisphere of $S(x_c,h)$, then
  there is $x \in P$ with $d(x,C) > h = d_H(P,P')$. Just take any $x$
  on the axis with $x_3>0$. The contradiction, seen also if $C_1<0$ or
  $C_2<0$, forces $C$ to be on the closed spherical triangle $T = \{ x
  \mid \abs{x} = 1, x_j \geq 0\}$.

  Now, no matter where $C \in T$ is, recalling that $P' \subset
  H_1^\complement \cap H_2^\complement \cap H_3^\complement \cap H_C$,
  it is easy to see that
  \[
  \sup_{A,B\in P\cup P'} \angle A x_c B < \pi
  \]
  and hence that $H_1^\complement \cap H_2^\complement \cap
  H_3^\complement \cap H_C$ fits inside an spherical cone that does
  not contain a plane. Moreover the minimal required angle of the
  spherical cone depends continuously on the location of $C \in
  T$. Compactness of the latter implies the claim.
\end{proof}

\section{Concluding remark}

In this paper, we establish two sharp quantitative results for the direct and inverse time-harmonic acoustic wave scattering problem. The first one is a logarithmic stability result in 
recovering the support of an inhomogeneous medium, independent of its contents, by a single far-field measurement, which quantifies the uniqueness result in \cite{HSV}. The second result 
shows that if an inhomogeneous medium possesses a corner, then it scatters an incident wave field stably in the sense that the energy of the corresponding scattered far-field possesses a positive
lower bound. This quantifies the corner scattering result in \cite{BPS} and has interesting implications to cloaking applications. Those topics are of fundamental importance in the wave scattering theory. In order to establish the quantitative results, we also make several technical new developments, which might be useful for tackling other direct and inverse scattering problems. Finally, we would like to remark that we only consider the case that the acoustic mediums are isotropic and it would be interesting and of practical importance to investigate the case that the inhomogeneous mediums are anisotropic. We are aware of a recent paper \cite{CX}, where the authors studied the acoustic scattering from an anisotropic acoustic medium that possesses a corner. It is shown that an anisotropic corner can always scatter a nontrivial far-field pattern, which extends the study in \cite{BPS} to the more challenging anisotropic case. The extension is technically highly nontrivial. It would be interesting to consider extending the quantitative studies in the current article to the anisotropic setting. We shall report our study in this aspect in our future work. 

Since the post of this work to arXiv in 2016, there have been many
developments in the literature on qualitatively and quantitatively
characterizing the geometrical singularities in wave scattering as
well as their implications to inverse problems and
invisibility. Accordingly, we mention here \cite{B,BLLW,BL,BL17,BL20,BL21,BLX}, \cite{CDL20,CDL21} as well
as a recent survey paper \cite{L20}.

\section*{Acknowledgement}

The authors would like to express their gratitudes to the anonymous
referee for constructive and insightful comments, which have led to
significant improvements on the presentation and results of this
paper.  The work of the first author was supported by the Academy of
Finland (decision no. 312124) and partly by a grant from the Estonian
Research Council (grant no. PRG 832).  The work of the second author
was supported by the startup fund and FRG grants from Hong Kong
Baptist University, and the Hong Kong RGC grants (projects 12302017
and 12301218).

%

\begin{thebibliography}{100}

\bibitem{AR} G. Alessandrini\ and\ L. Rondi, Determining a sound-soft
  polyhedral scatterer by a single far-field measurement,
  \textit{Proc. Amer. Math. Soc.} {\bf 133} (2005), no.~6, 1685--1691.

\bibitem{Behzadan--Holst} A. Behzadan\ and\ M. Holst, Multiplication
  in Sobolev spaces, revisited, \textit{Ark. Mat.} {\bf 59} (2021),
  no.~2, 275--306.

\bibitem{B} E. Bl{\aa}sten, Nonradiating sources and transmission
  eigenfunctions vanish at corners and edges, \textit{SIAM
    J. Math. Anal.} {\bf 50} (2018), no.~6, 6255--6270.

\bibitem{BLLW} E. Bl\aa sten, X. Li, H. Liu\ and\ Y. Wang, On
  vanishing and localizing of transmission eigenfunctions near
  singular points: a numerical study, \textit{Inverse Problems} {\bf
    33} (2017), no.~10, 105001, 24 pp.

\bibitem{BL} E. Bl\aa sten\ and\ Y.-H. Lin, Radiating and
  non-radiating sources in elasticity, \textit{Inverse Problems} {\bf
    35} (2019), no.~1, 015005, 16 pp.

\bibitem{BL17} E. Bl\aa sten\ and\ H. Liu, On vanishing near corners
  of transmission eigenfunctions, \textit{J. Funct. Anal.} {\bf 273}
  (2017), no.~11, 3616--3632.

\bibitem{BL20} E. Bl\aa sten\ and\ H. Liu, Recovering piecewise
  constant refractive indices by a single far-field pattern,
  \textit{Inverse Problems} {\bf 36} (2020), no.~8, 085005, 16 pp.

\bibitem{BL21} E. L. K. Bl\aa sten\ and\ H. Liu, Scattering by
  curvatures, radiationless sources, transmission eigenfunctions, and
  inverse scattering problems, \textit{SIAM J. Math. Anal.} {\bf 53}
  (2021), no.~4, 3801--3837.

\bibitem{BLX} E. Bl\aa sten, H. Liu\ and\ J. Xiao, On an
  electromagnetic problem in a corner and its applications,
  \textit{Anal. PDE} {\bf 14} (2021), no.~7, 2207--2224.

\bibitem{BPS} E. Bl\aa sten, L. P\"{a}iv\"{a}rinta\ and\ J. Sylvester,
  Corners always scatter, \textit{Comm. Math. Phys.} {\bf 331} (2014),
  no.~2, 725--753.

\bibitem{CX} F. Cakoni\ and\ J. Xiao, On corner scattering for
  operators of divergence form and applications to inverse scattering,
  (2019), preprint, arXiv:1905.02558.

\bibitem{CDL20} X. Cao, H. Diao\ and\ H. Liu, Determining a piecewise
  conductive medium body by a single far-field measurement,
  \textit{CSIAM Trans. Appl. Math.} {\bf 1} (2020), no.~4, 740--765.

\bibitem{CDL21} X. Cao, H. Diao\ and\ H. Liu, On the geometric
  structures of transmission eigenfunctions with a conductive boundary
  condition and applications, \textit{Comm. Partial Differential
    Equations} {\bf 46} (2021), no.~4, 630--679.

\bibitem{ctchan} H. Chen\ and\ C.~T. Chan, Acoustic cloaking in three
  dimensions using acoustic metamaterials, \textit{Applied Physics
    Letters}, {\bf 91} (2007), no.~18, 183518.

\bibitem{CY} J. Cheng\ and\ M. Yamamoto, Uniqueness in an inverse
  scattering problem within non-trapping polygonal obstacles with at
  most two incoming waves, \textit{Inverse Problems} {\bf 19} (2003),
  no.~6, 1361--1384.

\bibitem{CK} D. Colton\ and\ R. Kress, {\it Inverse acoustic and
  electromagnetic scattering theory}, second edition, Applied
  Mathematical Sciences, 93, Springer-Verlag, Berlin, 1998.

\bibitem{GT} D. Gilbarg\ and\ N. S. Trudinger, {\it Elliptic partial
  differential equations of second order}, second edition, Grundlehren
  der mathematischen Wissenschaften, 224, Springer-Verlag, Berlin,
  1983.

\bibitem{GKLU5} A. Greenleaf, Y. Kurylev , M. Lassas\ and\ G. Uhlmann,
  Cloaking devices, electromagnetic wormholes, and transformation
  optics, \textit{SIAM Rev.} {\bf 51} (2009), no.~1, 3--33.

\bibitem{GKLU4} A. Greenleaf, Y. Kurylev , M. Lassas\ and\ G. Uhlmann,
  Invisibility and inverse problems,
  \textit{Bull. Amer. Math. Soc. (N.S.)} {\bf 46} (2009), no.~1,
  55--97.

\bibitem{GLU2} A. Greenleaf, M. Lassas\ and\ G. Uhlmann, On
  nonuniqueness for Calder\'{o}n's inverse problem,
  \textit{Math. Res. Lett.} {\bf 10} (2003), no.~5-6, 685--693.

\bibitem{Hormander} L. H\"{o}rmander, {\it The analysis of linear
  partial differential operators. II}, Grundlehren der mathematischen
  Wissenschaften, 257, Springer-Verlag, Berlin, 1983.

\bibitem{HSV} G. Hu, M. Salo\ and\ E. V. Vesalainen, Shape
  identification in inverse medium scattering problems with a single
  far-field pattern, \textit{SIAM J. Math. Anal.} {\bf 48} (2016),
  no.~1, 152--165.

\bibitem{Isakov92} V. Isakov, Stability estimates for obstacles in
  inverse scattering, \textit{J. Comput. Appl. Math.} {\bf 42} (1992),
  no.~1, 79--88.

\bibitem{Isakov93} V. Isakov, New stability results for soft obstacles
  in inverse scattering, \textit{Inverse Problems} {\bf 9} (1993),
  no.~5, 535--543.

\bibitem{Isa} V. Isakov, {\it Inverse problems for partial
  differential equations}, second edition, Applied Mathematical
  Sciences, 127, Springer, New York, 2006.

\bibitem{KRS} C. E. Kenig, A. Ruiz\ and\ C. D. Sogge, Uniform Sobolev
  inequalities and unique continuation for second order constant
  coefficient differential operators, \textit{Duke Math. J.} {\bf 55}
  (1987), no.~2, 329--347.

\bibitem{Leo} U. Leonhardt, Optical conformal mapping,
  \textit{Science} {\bf 312} (2006), no.~5781, 1777--1780.

\bibitem{J53} J. Li, H. Liu, L. Rondi\ and\ G. Uhlmann, Regularized
  transformation-optics cloaking for the Helmholtz equation: from
  partial cloak to full cloak, \textit{Comm. Math. Phys.} {\bf 335}
  (2015), no.~2, 671--712.

\bibitem{L20} H. Liu, On local and global structures of transmission
  eigenfunctions and beyond, \textit{Journal of Inverse and Ill-posed
    Problems} (2020), 000010151520200099.
 
\bibitem{J23} H. Liu, Virtual reshaping and invisibility in obstacle
  scattering, \textit{Inverse Problems} {\bf 25} (2009), no.~4,
  045006, 16 pp.

\bibitem{LPRX} H. Liu, M. Petrini , L. Rondi\ and\ J. Xiao, Stable
  determination of sound-hard polyhedral scatterers by a minimal
  number of scattering measurements, \textit{J. Differential
    Equations} {\bf 262} (2017), no.~3, 1631--1670.

\bibitem{J33} H. Liu\ and\ H. Sun, Enhanced near-cloak by FSH lining,
  \textit{J. Math. Pures Appl. (9)} {\bf 99} (2013), no.~1, 17--42.

\bibitem{LU} H. Liu\ and\ G. Uhlmann, Regularized
  transformation-optics cloaking in acoustic and electromagnetic
  scattering, in {\it Inverse problems and imaging}, 111--136,
  Panor. Synth\`eses, 44, Soc. Math. France, Paris, 2015.
  
\bibitem{J18} H. Liu\ and\ J. Zou, Uniqueness in determining multiple
  polygonal scatterers of mixed type, \textit{Discrete
    Contin. Dyn. Syst. Ser. B} {\bf 9} (2008), no.~2, 375--396.

\bibitem{LZ} H. Liu\ and\ J. Zou, Uniqueness in an inverse acoustic
  obstacle scattering problem for both sound-hard and sound-soft
  polyhedral scatterers, \textit{Inverse Problems} {\bf 22} (2006),
  no.~2, 515--524.

\bibitem{Nachman88} A. I. Nachman, Reconstructions from boundary
  measurements, \textit{Ann. of Math. (2)} {\bf 128} (1988), no.~3,
  531--576.

\bibitem{PSV} L. P\"{a}iv\"{a}rinta, M. Salo\ and\ E. V. Vesalainen,
  Strictly convex corners scatter, \textit{Rev. Mat. Iberoam.} {\bf
    33} (2017), no.~4, 1369--1396.

\bibitem{PenSchSmi} J. B. Pendry, D. Schurig\ and\ D. R. Smith,
  Controlling electromagnetic fields, \textit{Science} {\bf 312}
  (2006), no.~5781, 1780--1782.

\bibitem{RU1} Rakesh\ and\ G. Uhlmann, Uniqueness for the inverse
  backscattering problem for angularly controlled potentials,
  \textit{Inverse Problems} {\bf 30} (2014), no.~6, 065005, 24 pp.

\bibitem{RU2} Rakesh\ and\ G. Uhlmann, The point source inverse
  back-scattering problem, in {\it Analysis, complex geometry, and
    mathematical physics: in honor of Duong H. Phong}, 279--289,
  Contemp. Math., 644, Amer. Math. Soc., Providence, RI, 2015.

\bibitem{Rondi08} L. Rondi, Stable determination of sound-soft
  polyhedral scatterers by a single measurement, \textit{Indiana
    Univ. Math. J.} {\bf 57} (2008), no.~3, 1377--1408.

\bibitem{RondiSini} L. Rondi\ and\ M. Sini, Stable determination of a
  scattered wave from its far-field pattern: the high frequency
  asymptotics, \textit{Arch. Ration. Mech. Anal.} {\bf 218} (2015),
  no.~1, 1--54.

\bibitem{Ruiz} A. Ruiz, {\it Harmonic Analysis and Inverse Problems,
  Lecture Notes}, 2002, accessed 24.01.2018.

\bibitem{Sylvester--Uhlmann} J. Sylvester\ and\ G. Uhlmann, A global
  uniqueness theorem for an inverse boundary value problem,
  \textit{Ann. of Math. (2)} {\bf 125} (1987), no.~1, 153--169.

\bibitem{Triebel1} H. Triebel, {\it Theory of function spaces},
  Monographs in Mathematics, 78, Birkh\"{a}user Verlag, Basel, 1983.

\bibitem{Triebel2} H. Triebel, {\it Theory of function spaces. II},
  Monographs in Mathematics, 84, Birkh\"{a}user Verlag, Basel, 1992.

\bibitem{Uhl2} G. Uhlmann, Visibility and invisibility, in {\it ICIAM
  07---6th International Congress on Industrial and Applied
  Mathematics}, 381--408, Eur. Math. Soc., Z\"{u}rich, 2009.

\bibitem{Uhl} G. Uhlmann, editor, {\it Inverse problems and
  applications: inside out. II}, Mathematical Sciences Research
  Institute Publications, 60, Cambridge University Press, Cambridge,
  2013.

\end{thebibliography}

\end{document}